\documentclass[a4paper,12pt]{article}

\usepackage{graphicx,fancyhdr,mathrsfs}
\usepackage{amsfonts,amsmath,amssymb,amsthm,mathtools}
\usepackage{multicol,multirow}
\usepackage{url}

\usepackage[colorlinks=true,citecolor=blue,urlcolor=blue,linkcolor=blue]{hyperref}
\usepackage{indentfirst}
\usepackage{color,xcolor,comment}
\usepackage[authoryear,round]{natbib}
\usepackage[toc,page]{appendix}
\usepackage{subfig}
\usepackage[top=20mm,bottom=30mm,left=20mm,right=20mm]{geometry}
\usepackage{enumitem}
\usepackage{array,booktabs}
\usepackage{csquotes}
\usepackage[compact]{titlesec}
\usepackage{bm,bbm}
\usepackage[onehalfspacing]{setspace}

\setlist{
  align=left,
  labelindent=0mm,
  leftmargin=!,
  itemindent=0mm,
  listparindent=\parindent,
  parsep=0mm,
  topsep=1mm,
  itemsep=1mm
}

\theoremstyle{definition}
\newtheorem{theorem}{Theorem}[section]
\newtheorem{proposition}[theorem]{Proposition}
\newtheorem{lemma}[theorem]{Lemma}
\newtheorem{corollary}[theorem]{Corollary}
\newtheorem{definition}[theorem]{Definition}
\newtheorem{example}[theorem]{Example}
\newtheorem{remark}[theorem]{Remark}

\newcommand{\R}{\mathbb{R}}
\newcommand{\E}{\mathbb{E}}

\newcommand{\Cop}{\mathcal{C}}
\newcommand{\dd}{\,\mathrm{d}}
\newcommand{\eps}{\varepsilon}

\newcommand{\sgn}{\operatorname{sign}}
\newcommand{\Meas}{\mathfrak{M}}
\newcommand{\id}{\mathbbm{1}}

\title{A Characterization of Measures of Concordance\\ of Degree Two}

\author{
  Takaaki Koike\thanks{\protect\linespread{1}\protect\selectfont Corresponding author.}\,\,\thanks{\protect\linespread{1}\protect\selectfont
    Graduate School of Economics, Hitotsubashi University, 2-1, Naka, Kunitachi, Tokyo 186-8601, Japan.
    Email: \texttt{takaaki.koike@r.hit-u.ac.jp}}\, and
  Haruki Tsunekawa\thanks{\protect\linespread{1}\protect\selectfont
    Graduate School of Economics, Hitotsubashi University, 2-1, Naka, Kunitachi, Tokyo 186-8601, Japan.
    Email: \texttt{em265009@g.hit-u.ac.jp}}
}

\begin{document}

\maketitle

\begin{abstract}
We characterize bivariate measures of concordance of degree at most two in the sense that the measures evaluated at copula-mixture segments are polynomials of degree at most two.
Our main device is the canonical polarization, which converts quadraticity of a functional into separate affinity, thereby allowing an integral representation for affine functionals to be applied sectionwise.
Using this technique, we prove that measures of degree at most two are in bijection with
copula-indexed measure fields satisfying several properties including square-group equivariance. The field at independence determines the affine
linearization, while its displacement determines the purely quadratic
part.
Moreover, exact degree two is equivalent to field non-constancy. 
Leveraging our characterization result, we provide a construction of degree-two measures of concordance based on affine copula operators that
are square-group equivariant.
As a more concrete example, we derive a necessary and sufficient condition for a weighted averaging operator of square-group transformed copulas to generate a measure of concordance of degree two.
\end{abstract}

\noindent\emph{MSC classification:}
62H20; % Measures of association
62H05; % Characterization and structure theory
60E15. % Inequalities; stochastic orderings
\\
\noindent\emph{Keywords:}
copula; 
degree; 
Kendall's tau;
measure of concordance;
Spearman's rho.

\section{Introduction}

A bivariate measure of concordance assigns a real number to a copula in a
way that respects the concordance order and the symmetries of the unit
square. The axiomatic properties are presented in~\citet{scarsini1984}; standard
references include \citet{nelsen2006} and \citet{durantesempi2016}.
Classical examples include Spearman's rho, Blomqvist's beta, Gini's gamma,
and Kendall's tau. 
Their behavior under copula-mixtures provides a structural distinction: the first three are affine, whereas Kendall's tau is quadratic.

For the affine case, where such measures of concordance are said to be of \emph{degree one}, \citet{edwardstaylor2009} established various characterizations of all such measures.
One such characterization is based on square-group invariant measures as follows:
\begin{align}\label{eq:intro:deg:1:characterization}
    \kappa(C)=\int_{[0,1]^2}(C-\Pi)\,\dd
\nu,\qquad C\in \Cop,
\end{align}
where $\kappa$ is a measure of concordance, $\Cop$ is the set of all bivariate copulas and $\nu$ is a measure on $[0,1]^2$ satisfying some invariant properties.
Related affine
classes include the copula-induced measures of
\citet{fuchs2016induced} and the correlation-based measures characterized
by \citet{hofertkoike2019,koike2023matrix}.

Quadratic measures of concordance are said to have \emph{degree two}, where Kendall's tau is the primary example.
Other instances can be found, for example, in~\citet{manstavicius2022,manstavicius2024,manstavicius2025} and the mutual
association measures of~\citet{borroni2019}.
In contrast to the case of degree-one, a complete characterization of measures of concordance of degree two is not known in the literature to the best of our knowledge.
This characterization problem is posed in~\citet{taylor2007}, and
 revisited, for example, by~\citet{edwardstaylor2009,manstavicius2022,manstavicius2025}.
 
We address this problem via what we call the \emph{canonical polarization}:
\begin{align*}
  b_\kappa(C,D)
  &=
  2\kappa\left(\frac{C+D}{2}\right)
  -\frac{\kappa(C)+\kappa(D)}{2},\qquad C,D \in \Cop,
\end{align*}
where we show its properties and connections to measures of concordance in Section~\ref{sec:polarization}.
In particular, the canonical polarization $b_{\kappa}$ captures the structure of the associated measure of concordance of degree two via $b_\kappa(C,C)
  =
  \kappa(C)$ and 
\begin{align*}
  \kappa((1-t)C+tD)
  &=
  (1-t)^2\kappa(C)
  +2t(1-t)b_\kappa(C,D)
  +t^2\kappa(D),\qquad t\in[0,1],
\end{align*}
where $C,D\in\Cop$.
In addition, the canonical polarization converts quadraticity of $\kappa$ into separate affinity, which allows an integral representation of affine functionals to be applied sectionwise.
Technically, all axiomatic properties and quadraticity of a measure of concordance are equivalently stated by the form $b:\Cop\times \Cop\to \R$.
In this context, our main characterization is stated by the canonical polarization as follows:
\begin{align}\label{eq:intro:characterization}
  b(C,D)
  =
  \int_{(0,1)^2}(C-\Pi)\,\dd\nu_\Pi
  +
  \int_{(0,1)^2}(D-\Pi)\,\dd\nu_C,
\end{align}
where $\nu=(\nu_C)_{C\in \Cop}$ is a measure field satisfying several properties including square-group equivariance.
More precisely, every measure of concordance of degree at most two has canonical polarization of the form~\eqref{eq:intro:characterization}.
Conversely, the canonical polarization~\eqref{eq:intro:characterization} yields a measure of concordance of degree at most two.
Moreover, as seen by comparing~\eqref{eq:intro:characterization} with~\eqref{eq:intro:deg:1:characterization}, the measure is degree one if and only if the field is constant : $\nu_C =(1/2)\,\nu$ for all $C\in \Cop$.
For Kendall's
tau, the corresponding field is $\nu_C^\tau=4\mu_C$, where $\mu_C$ is the measure associated with $C$.

To leverage the characterization~\eqref{eq:intro:characterization}, we provide a construction of measures of concordance of degree two based on  the biconvex form of copulas $Q(C,D)
  =
  4\int_{[0,1]^2}C\,\dd\mu_D-1$ and a copula operator $\Theta:\Cop\to \Cop$.
As a more concrete construction, we consider the class of copula-to-copula operators given by weighted averages of all square-group transforms of a copula.
We then derive a necessary and sufficient condition on the weights of transforms such that the averaging operator generates a measure of concordance of degree two.

The remainder of the paper is organized as follows.
Section~\ref{sec:preliminaries} reviews square-group action, measures
of concordance and the notion of degree.
Section~\ref{sec:main:characterization} develops the canonical polarization
and the characterization by measure fields.
Section~\ref{sec:constructions} then presents the copula-operator construction and studies the weighted averaging operator.
Section~\ref{sec:conclusion} concludes. All proofs are collected in the
Appendix.

\section{Preliminaries}\label{sec:preliminaries}

Let $\Cop$ denote the set of all bivariate copulas, equipped throughout
with the uniform topology. It is a compact convex subset of the set of continuous functions on
$[0,1]^2$. 
% On $\Cop$, pointwise and uniform convergence are equivalent, and they are also equivalent to weak convergence of the associated copula measures; see~\citet{durantesempi2016}. 
We write $M$, $W$, and $\Pi$ for
the comonotone, countermonotone, and independence copulas, respectively.
The pointwise order on $\Cop$ is the bivariate concordance order and is
denoted by $\preceq$.

Let $\mathcal D_4$ be the symmetry group of $[0,1]^2$ generated by the coordinate transposition $\pi(u,v)=(v,u)$ and the first-coordinate reflection $\sigma_1(u,v)=(1-u,v)$.
Namely, 
\begin{align*}
    \mathcal{D}_4 &= \{ \mathrm{id},\, \pi,\, \sigma_1,\, \pi \circ \sigma_1 \circ \pi,\, \sigma_1 \circ \pi \circ \sigma_1 \circ \pi, \,\sigma_1 \circ \pi, \,\pi \circ \sigma_1, \,\sigma_1 \circ \pi \circ \sigma_1 \}\\
    &=
\{ \mathrm{id},\, \pi, \,\sigma_1, \,\sigma_2, \,\sigma_1 \circ \sigma_2,\, \sigma_1 \circ \pi,\, \pi \circ \sigma_1,\, \pi \circ \sigma_1 \circ \sigma_2 \},
\end{align*}
where $\sigma_2=\pi\circ\sigma_1\circ\pi$.
We also write $\varrho(u,v)=(1-u,1-v)=\sigma_1\circ \sigma_2$ in short.

%new
For $\gamma\in\mathcal D_4$, we define the copula $C^\gamma$
via the push-forward measure as
$\mu_{C^\gamma}=\gamma_\#\mu_C$, where
$\gamma_\#\mu_C(A)=\mu_C(\gamma^{-1}(A))$ for
$A\in\mathcal B([0,1]^2)$.
This is a continuous affine action on $\Cop$.  
The induced action satisfies $C^{\mathrm{id}}
  =
  C$ and 
  $
  (C^{\gamma_1})^{\gamma_2}
  =
  C^{\gamma_2\circ\gamma_1}$ for
$C\in\Cop$ and $\gamma_1,\gamma_2\in\mathcal D_4$.
Note that $C^\varrho(u,v)=u+v-1+C(1-u,1-v)$ is the survival copula of $C$.
For convenience, we introduce the map $\eps:\mathcal D_4\rightarrow\{-1,1\}$ determined  by $\eps(\pi)=1$ and $\eps(\sigma_1)=-1$ through $\eps(\gamma_1\circ \gamma_2)=\eps(\gamma_1)\,\eps( \gamma_2)$, $\gamma_1,\gamma_2\in\mathcal D_4$.
% Whenever $\mu_C$ is regarded as a measure on $(0,1)^2$, we identify it with its restriction to that open square; note that $\mu_C(\partial[0,1]^2)=0$ for the boundary mass.

We next provide the axiomatic definition of a measure of concordance following~\citet{scarsini1984}.

\begin{definition}[Measure of concordance]\label{def:moc}
A functional $\kappa:\Cop\to\R$ is a \emph{measure of concordance} if the following properties hold.
\begin{enumerate}[label=\textup{(A\arabic*)},leftmargin=*,widest=A5]
  \item\label{item:pi:sym} $\kappa(C^\pi)=\kappa(C)$ for every $C\in\Cop$.
  \item\label{item:monotone} $C\preceq D$ implies $\kappa(C)\leq\kappa(D)$.
  \item\label{item:M} $\kappa(M)=1$.
  \item\label{item:sigma:sym} $\kappa(C^{\sigma_1})=-\kappa(C)$ for every $C\in\Cop$.
  \item\label{item:conv} $C_n\to C$ uniformly implies $\kappa(C_n)\to\kappa(C)$.
\end{enumerate}
\end{definition}
Note that~\ref{item:pi:sym} and~\ref{item:sigma:sym} are equivalently summarized by $\kappa(C^\gamma)=\eps(\gamma)\kappa(C)$, $\gamma\in \mathcal D_4$.
Since $\Pi^{\sigma_1}=\Pi$, every measure of concordance satisfies $\kappa(\Pi)=0$.

Following the polynomial-type classification of
\citet{edwardstaylor2009}, we next define the degree of measures of concordance.

\begin{definition}[Degree]\label{def:mixture-degree}
A functional $\kappa:\Cop\to\R$ is said to have \emph{degree} at most two if, for every $C,D\in\Cop$, the map $ t\mapsto\kappa((1-t)C+tD)$, $ 0\leq t\leq1$, agrees on $[0,1]$ with a real polynomial of degree at most two.
We say that $\kappa$ has \emph{exact degree two} if $\kappa$ has degree at most two but is not affine on $\Cop$.
\end{definition}

\section{The characterization with measure-field}
\label{sec:main:characterization}

Our characterization is based on the canonical polarization, which is introduced in Section~\ref{sec:polarization}.
Section~\ref{sec:characterization:field} then provides the main characterization result with measure-field.
We also give
examples of corresponding measure-fields for existing measures of concordance of degree at most two.

\subsection{Canonical polarization}\label{sec:polarization}

Our characterization is based on the following object.

\begin{definition}[Canonical polarization]\label{def:canonical:polarization}
 For a functional $\kappa:\Cop\to\R$, we call
\begin{align*}
% \label{eq:main-polarization}
  b_\kappa(C,D)
  =2\kappa\left(\frac{C+D}{2}\right)
    -\frac{\kappa(C)+\kappa(D)}{2}
\end{align*}   
the \emph{canonical polarization} of $\kappa$.
\end{definition}

Properties of measures of concordance can be translated into those of canonical polarizations.
 We first establish the result on polarization of linewise polynomiality.  
Below, we say that a map $b:\Cop\times\Cop\to\R$ is \emph{(jointly) continuous} if it is continuous with respect to (w.r.t.) the product uniform topology.

\begin{proposition}[Canonical polarization]\label{prop:canonical-polarization}
Let $\kappa:\Cop\to\R$ be a continuous functional of degree at most two. 
Let $b_\kappa$ be the canonical polarization in Definition~\ref{def:canonical:polarization}.
Then the following properties hold.
\begin{enumerate}[label=\textup{(\roman*)},leftmargin=*,widest=ii]
  \item $b_\kappa$ is jointly continuous, symmetric, and affine in each argument.
  \item $b_\kappa(C,C)=\kappa(C)$ for every $C\in\Cop$.
  \item For every $C,D\in\Cop$ and $t\in[0,1]$,
  \begin{align}\label{eq:quadratic-segment}
    \kappa((1-t)C+tD)
    &=(1-t)^2\kappa(C)
      +2t(1-t)b_\kappa(C,D)
      +t^2\kappa(D).
  \end{align}
  \item $b_\kappa$ is the unique symmetric separately affine map on $\Cop\times\Cop$ whose diagonal is $\kappa$.
\end{enumerate}
\end{proposition}

We next translate affinity of $\kappa$ with its canonical polarization.
If $\kappa$ is affine, we have from the definition of $b_\kappa$ that
\begin{align}\label{eq:affine-polarization}
  b_\kappa(C,D)=\frac{\kappa(C)+\kappa(D)}{2},
  \qquad C,D\in\Cop.
\end{align}
Conversely, if
\eqref{eq:affine-polarization} holds, then substituting this into
\eqref{eq:quadratic-segment} gives affinity $  \kappa((1-t)C+tD)
  =
  (1-t)\kappa(C)+t\kappa(D)$ for every $C,D\in\Cop$ and
$t\in[0,1]$.
Consequently, we obtain the following result.

\begin{corollary}[Exact degree]\label{cor:exact-degree}
Suppose that the assumptions of
Proposition~\ref{prop:canonical-polarization} hold.
Then $\kappa$ is affine if and only if
\eqref{eq:affine-polarization} holds. 
Consequently, $\kappa$ has exact degree two if and only if \eqref{eq:affine-polarization} fails for at least one pair $(C,D)$.
\end{corollary}

In terms of the canonical polarization, axiomatic properties of measures of concordance can be reformulated as follows.

\begin{proposition}[Characterization in terms of the canonical polarization]
\label{prop:characterization:b}
Let $\kappa:\Cop\to\R$ be a continuous functional of degree at most two.
Then $\kappa$ is a measure of concordance if and only if the following
conditions hold.
\begin{enumerate}[label=\textup{(B\arabic*)},leftmargin=*,widest=B3]
  \item\label{cond:B1}
  $b_\kappa(C^\gamma,D^\gamma)
  =\eps(\gamma)b_\kappa(C,D)$ for all
  $C,D\in\Cop$ and $\gamma\in\mathcal D_4$.

  \item\label{cond:B2}
  For every $C\in\Cop$, the section
  $D\mapsto b_\kappa(C,D)$ is $\preceq$-monotone.

  \item\label{cond:B3}
  $b_\kappa(M,M)=1$.
\end{enumerate}
\end{proposition}

For each $d\in\{1,2\}$, denote by $\mathcal K_d$ the class of all
measures of concordance of degree exactly $d$ and by $\mathcal B_d$ the class of all jointly continuous, symmetric and separately affine maps $b:\Cop\times\Cop\to\R$ satisfying 
\ref{cond:B1}--\ref{cond:B3} with $b$ in place of $b_\kappa$, and $C\mapsto b(C,C)$ has degree exactly $d$. Write $\mathcal K_{1,2}
  =
  \mathcal K_1\cup\mathcal K_2$ and $\mathcal B_{1,2}
  =
  \mathcal B_1\cup\mathcal B_2$.
Combining the above results in this subsection, we obtain the following.

\begin{proposition}[Bijection induced by canonical polarization]
\label{prop:canonical-polarization-bijection}
The canonical-polarization map $\mathsf P:\mathcal K_{1,2}
  \rightarrow
  \mathcal B_{1,2}$ defined by $\mathsf P(\kappa)=
  b_\kappa$
is a bijection.
Its inverse is the diagonal map $\mathsf D:\mathcal B_{1,2}
  \rightarrow
  \mathcal K_{1,2}$ defined by $\mathsf D(b): C\rightarrow 
  b(C,C)$.
Moreover, for each $d\in\{1,2\}$, the restriction of $\mathsf P$
to $\mathcal K_d$ is a bijection from $\mathcal K_d$ onto
$\mathcal B_d$.
\end{proposition}

\subsection{The characterization theorem}\label{sec:characterization:field}

 Denote by $w$ the pointwise Fr\'echet--Hoeffding order width, that is,
\begin{align*}
  w(u,v)
  =(M-W)(u,v)
  =\min\{u,v,1-u,1-v\},
  \qquad (u,v)\in[0,1]^2.
\end{align*}
Note that $w$ is continuous on the closed square, positive on its
interior, and zero on its boundary. 

Let $\Meas_w^+$ denote the convex cone
of non-negative Borel measures $\nu$ on $(0,1)^2$
satisfying $\int_{(0,1)^2}w\,\dd\nu
  <\infty$.
Since $w$ is bounded away from zero on every non-empty compact
subset $K$ of $(0,1)^2$, every $\nu\in\Meas_w^+$ satisfies $\nu(K)\leq
  \int_{(0,1)^2}w\,\dd\nu/\min_K w
  <
  \infty$.
Thus every measure in $\Meas_w^+$ is finite on compact subsets of
$(0,1)^2$ and hence is locally finite.
Since $(0,1)^2$ is a locally compact Hausdorff space in which every
open set is $\sigma$-compact, every such measure is a regular Borel
measure~\citep[Theorem~2.18]{rudin1987}.
The Fr\'echet--Hoeffding bounds imply $|C-\Pi|\leq w$, and thus every integral
below is finite unless stated explicitly.
Note that measures in $\Meas_w^+$ need not have finite total mass. Nevertheless, if
$\nu,\widetilde\nu\in\Meas_w^+$ and $|f|\leq K w$ for some $K>0$, the integral $\int f\,\dd(\nu-\widetilde\nu)$ is well-defined as the difference $\int f\,\dd\nu-\int f\,\dd\widetilde\nu$ since both integrals are finite. 
% Note that $\nu-\tilde \nu$ is a signed Radon measure on every compact subset of $(0,1)^2$.

Our characterization is based on a field of $\Meas_w^+$ satisfying certain properties listed below.

\begin{definition}[Admissible field]\label{def:admissible:field}
A field $\nu=(\nu_C)_{C\in \mathcal C}$ of measures is called \emph{admissible} if the following six properties hold.
    \begin{enumerate}[label=\textup{(P\arabic*)},leftmargin=*,widest=ii]
  \item\label{cond:positive:w:field} $\nu_C\in\Meas_w^+$ for every $C\in\Cop$.
  \item\label{cond:affine} The field is affine, that is, $\nu_{(1-t)C+tD}=(1-t)\nu_C+t\nu_D$, $C,D\in\Cop$, $0\leq t\leq1$.
  \item\label{cond:reciprocity} The following 
  \emph{reciprocity identity} holds:
  \begin{align}\label{eq:reciprocity:identity}
    \int_{(0,1)^2}(D-\Pi)\,\dd(\nu_C-\nu_\Pi)
    =\int_{(0,1)^2}(C-\Pi)\,\dd(\nu_D-\nu_\Pi),\quad C,D \in \Cop.
  \end{align}
  \item\label{cond:D4:equivariance} The field is equivariant under $\mathcal D_4$: $\nu_{C^\gamma}=\gamma_\#\nu_C$
    , $C\in\Cop$, $\gamma\in \mathcal D_4$.
  \item\label{cond:diagonal:continuity} The diagonal map $\zeta_\nu:\Cop\to \R$:
  \begin{align}\label{eq:diagonal:map}
    \zeta_\nu(C)= \int_{(0,1)^2}(C-\Pi)\,\dd\nu_\Pi
      +\int_{(0,1)^2}(C-\Pi)\,\dd\nu_C
  \end{align}
  is continuous on $\Cop$.
  \item\label{cond:normalization} The field is normalized by
  \begin{align*}
    \int_{(0,1)^2}(M-\Pi)\,\dd\nu_\Pi
    +\int_{(0,1)^2}(M-\Pi)\,\dd\nu_M=1.
  \end{align*}
\end{enumerate}
\end{definition}

\begin{theorem}[Characterization with measure-field]
\label{thm:measure:field:characterization}
The following two assertions hold.
\begin{enumerate}[label=\textup{(\roman*)},leftmargin=*,widest=ii]
  \item\label{item:characterization:if} Let $\kappa$ be a measure of concordance of degree at most
  two.  Then there is a unique admissible field $\nu^\kappa=(\nu_C^\kappa)_{C\in\Cop}$ such that
  \begin{align}\label{eq:anchored:field:representation}
    b_\kappa(C,D)
    =\int_{(0,1)^2}(C-\Pi)\,\dd\nu_\Pi^\kappa
     +\int_{(0,1)^2}(D-\Pi)\,\dd\nu_C^\kappa.
  \end{align}
  \item\label{item:characterization:only:if}  Conversely, let $\nu=(\nu_C)_{C\in\Cop}$ be an admissible field, and define $\kappa_\nu$ as the diagonal map $\zeta_\nu$ in~\eqref{eq:diagonal:map}.
  Then $\kappa_\nu$ is a measure of concordance of degree at most
  two.
 In addition, its canonical polarization is the right-hand side of
  \eqref{eq:anchored:field:representation}, and the corresponding admissible field specified in~\ref{item:characterization:if}  is
  $\nu$.
\end{enumerate}
\end{theorem}

Consequently, there is a one-to-one correspondence between measures of concordance of degree at most two and admissible fields via $\kappa\mapsto\nu^\kappa$ and $\nu\mapsto\kappa_\nu$.

% \subsection{Examples}\label{sec:examples}

We record how the characterization specializes to existing measures of concordance.

\begin{example}[Affine measures of concordance]\label{ex:affine}
Suppose that $\kappa$ is an affine measure of concordance.
According to~\citet[][Theorem~1]{edwardstaylor2009}, it admits the representation
 $\kappa(C)=\int_{(0,1)^2}(C-\Pi)\,\dd\nu_\kappa$, where $\nu_\kappa$ is the unique measure on $([0,1]^2,\mathcal B([0,1]^2))$ such that $\nu_\kappa(\partial[0,1]^2)=0$, $\gamma_\#\nu_{\kappa}=\nu_{\kappa}$ (\emph{$\mathcal D_4$-invariance}) and $\int (M-\Pi)\,\dd\nu_\kappa=1$.
By comparing~\eqref{eq:affine-polarization} with~\eqref{eq:anchored:field:representation} in terms of Corollary~\ref{cor:exact-degree}, the unique admissible field of $\kappa$ is constant and given by $(\nu_C^\kappa)_{C\in\Cop}=((1/2)\,\nu_\kappa|_{(0,1)^2})$; it is straightforward to check that this field is admissible.
\end{example}

\begin{example}[Kendall's tau]\label{ex:kendall}
Let $[C,D]=\int_{[0,1]^2} C\,\dd\mu_D$ and $Q(C,D)=4[C,D]-1$.
Note that \emph{Kendall's tau} is given by $\tau(C)=Q(C,C)$.
The map $Q$ on $\mathcal C \times \mathcal C$ is jointly continuous, symmetric, and separately affine; see~\citet{fuchs2016biconvex}.
Consequently, $Q$ is the canonical polarization of $\tau$ by Proposition~\ref{prop:canonical-polarization}~\textup{(iv)}.
Since
\begin{align*}
  Q(C,D)
=\left\{\int_{(0,1)^2}\Pi\,\dd(4\mu_C|_{(0,1)^2})-1\right\}
    +\int_{(0,1)^2}(D-\Pi)\,\dd(4\mu_C|_{(0,1)^2}),
\end{align*}
comparison with~\eqref{eq:anchored:field:representation} shows that the
unique admissible field associated with Kendall's tau is given by $(\nu_C^\tau)=(4\mu_C|_{(0,1)^2})$.
The admissibility of this field follows directly from the properties established in~\citet{fuchs2016biconvex}.
\end{example}

\section{Constructions with copula operators}\label{sec:constructions}

In this section, we construct measures of concordance of degree two by using the characterization result of Theorem~\ref{thm:measure:field:characterization}.

The first result converts an operator on copulas directly into an admissible
measure field. It provides a practical sufficient condition for all six
requirements in Definition~\ref{def:admissible:field}.

\begin{proposition}[Self-adjoint copula-operator construction]
\label{prop:operator-construction}
Let $\Theta:\Cop\to\Cop$ be continuous and affine. Assume that $\Theta$
satisfies the following properties.
\begin{enumerate}[label=\textup{(S\arabic*)},leftmargin=*,widest=S3]
  \item\label{cond:T-equivariant}
  $\Theta(C^\gamma)=\Theta(C)^\gamma$ for every
  $C\in\Cop$ and $\gamma\in\mathcal D_4$.

  \item\label{cond:T-fixed}
  $\Theta(\Pi)=\Pi$ and $\Theta(M)=M$.

  \item\label{cond:T-self-adjoint}
  $Q(\Theta(C),D)=Q(C,\Theta(D))$ for every $C,D\in\Cop$.
\end{enumerate}
Then $\nu^\Theta=(\nu_C^\Theta)_{C\in\Cop}$ with $\nu_C^\Theta
  =
  4\mu_{\Theta(C)}$ is an admissible field. The corresponding measure of concordance and its
canonical polarization are $\kappa_\Theta(C)
  =
  Q(\Theta(C),C)$ and $b_{\kappa_\Theta}(C,D)
  =
  Q(\Theta(C),D)$, respectively. Moreover, $\kappa_\Theta$ has exact degree two.
\end{proposition}

The square-group action and averaging over square-group transforms are
established tools as seen, for example, in~\citet{fuchsschmidt2014,fuchs2016induced}.
In this viewpoint, we specialize Proposition~\ref{prop:operator-construction} to convex
averages of the eight square-group transforms of a copula.

\begin{proposition}[Weighted square-group averages]
\label{prop:group:averaging}
For a vector
$\bm w=(w_\gamma)_{\gamma\in\mathcal D_4}$ with
$w_\gamma\geq0$ for every $\gamma\in\mathcal D_4$ and
$\sum_{\gamma\in\mathcal D_4}w_\gamma=1$, define the operator
$\Theta_{\bm w}:\Cop\to\Cop$ by
\begin{align}
  \Theta_{\bm w}(C)
  &=
  \sum_{\gamma\in\mathcal D_4}
  w_\gamma C^\gamma,
  \qquad C\in\Cop.
  \label{eq:group-average-operator}
\end{align}
Then $\Theta_{\bm w}$ satisfies
Conditions~\ref{cond:T-equivariant}--\ref{cond:T-self-adjoint} of
Proposition~\ref{prop:operator-construction} if and only if $w_\pi
  =
  w_{\pi\circ \varrho}$ and $w_{\sigma_1}
  =
  w_{\sigma_2}
  =
  w_{\sigma_1\circ\pi}
  =
  w_{\pi\circ\sigma_1}
  =
  0$ hold.
\end{proposition}

When $ w_{\mathrm{id}}=1$, the resulting measure $\kappa_{\Theta_{\bm w}}$ is Kendall's tau.
When $w_{\varrho}=1$, this radial-cross operation yields the measure of concordance $\kappa_{\mathrm{rc}}(C)=Q(C^{\varrho},C)$.
Finally, when $w_{\pi}=1/2$, and thus $w_{\pi\circ \varrho}=1/2$, this diagonal-cross operation leads to  $\kappa_{\mathrm{dc}}(C)
  =
  \{
    Q(C^\pi,C)+Q(C^{\pi\circ \varrho},C)
\}/2$.
By relabeling the weights by $w_\tau
  =
  w_{\mathrm{id}}$, $w_{\mathrm{rc}}
  =
  w_\varrho$ and $w_{\mathrm{dc}}
  =
  w_\pi+w_{\pi\circ \varrho}
  =
  2w_\pi
  =
  2w_{\pi\circ \varrho}$,~\eqref{eq:group-average-operator} yields
\begin{align}
  \Theta_{\boldsymbol w}(C)
  =
  w_\tau C
  +
  w_{\mathrm{rc}}C^\varrho
  +
 w_{\mathrm{dc}} \frac{
    C^\pi+C^{\pi\circ \varrho}
 }{2},
  \qquad
w_\tau,w_{\mathrm{rc}},w_{\mathrm{dc}}\geq0,
\qquad
w_\tau+w_{\mathrm{rc}}+w_{\mathrm{dc}}=1.
  \label{eq:central-orbit-simplex}
\end{align}
The measure of concordance associated with
\eqref{eq:central-orbit-simplex} is
\begin{align*}
  \kappa_{\boldsymbol w}(C)
  =
  w_\tau\tau(C)
  +
  w_{\mathrm{rc}}\kappa_{\mathrm{rc}}(C)
  +
  w_{\mathrm{dc}}\kappa_{\mathrm{dc}}(C).
  % \label{eq:central-orbit-MOC-simplex}
\end{align*}
This measure has exact degree two by
Proposition~\ref{prop:operator-construction}; indeed, for $C_t:=(1-t)\Pi+tM$, we have that $\Theta_{\boldsymbol w}(C_t)
  =
  C_t$ and $\kappa_{\boldsymbol w}(C_t)
  =
  (2/3)\,t+(1/3)\,t^2$.
Its canonical polarization and admissible field are, respectively,
\begin{align*}
  b_{\kappa_{\boldsymbol w}}(C,D)
  &=
  w_\tau Q(C,D)
  +
  w_{\mathrm{rc}}Q(C^\varrho,D)
 +w_{\mathrm{dc}}
  \frac{Q(C^\pi,D)+Q(C^{\pi\circ \varrho},D)}{2}\\
  \nu_C^{\boldsymbol w}
  &=
  4w_\tau\mu_C
  +
  4w_{\mathrm{rc}}\mu_{C^\varrho}
  +
  4w_{\mathrm{dc}}\frac{
    \mu_{C^\pi}+\mu_{C^{\pi\circ \varrho}}
 }{2}.
\end{align*}

\begin{remark}[Special cases]\label{rem:restriction-identities}
If $C$ is exchangeable, that is, $C=C^\pi$, then
  \begin{align*}
    \Theta_{\boldsymbol w}(C)
    &=
    \left(
      w_\tau+\frac{w_{\mathrm{dc}}}{2}
    \right)C
    +
    \left(
      w_{\mathrm{rc}}+\frac{w_{\mathrm{dc}}}{2}
    \right)C^\varrho,\\
    \kappa_{\boldsymbol w}(C)
    &=
    \left(
      w_\tau+\frac{w_{\mathrm{dc}}}{2}
    \right)\tau(C)
    +
    \left(
      w_{\mathrm{rc}}+\frac{w_{\mathrm{dc}}}{2}
    \right)\kappa_{\mathrm{rc}}(C).
  \end{align*}
If $C$ is radially symmetric, that is, $C=C^\varrho$, then
  \begin{align*}
    \Theta_{\boldsymbol w}(C)
    &=
    (w_\tau+w_{\mathrm{rc}})C+w_{\mathrm{dc}}C^\pi,
  \\
    \kappa_{\boldsymbol w}(C)
    &=
    (w_\tau+w_{\mathrm{rc}})\tau(C)
    +
    w_{\mathrm{dc}}Q(C^\pi,C).
  \end{align*}
Consequently, if $C$ is both exchangeable and radially symmetric, then $\kappa_{\mathrm{dc}}(C)
  =
  \kappa_{\mathrm{rc}}(C)
  =
  \tau(C)$, and thus $\Theta_{\boldsymbol w}(C)
    =
    C$ and $\kappa_{\boldsymbol w}(C)
    =
    \tau(C)$.
\end{remark}

Let $(U_1,V_1)$ and $(U_2,V_2)$ be independent copies of $(U,V)\sim C$.
Then 
\begin{align*}
  Q(C,C)
  &=
  \E\left[
    \sgn\{(U_1-U_2)(V_1-V_2)\}
  \right],\\
  Q(C^\varrho,C)
  &=
  \E\left[
    \sgn\{(U_1+U_2-1)(V_1+V_2-1)\}
  \right],\\
  Q(C^\pi,C)
  &=
  \E\left[
    \sgn\{(V_1-U_2)(U_1-V_2)\}
  \right],\\
  Q(C^{\pi\circ \varrho},C)
  &=
  \E\left[
    \sgn\{(V_1+U_2-1)(U_1+V_2-1)\}
  \right].
\end{align*}
Therefore, we obtain the following kernel representation of $\kappa_{\bm w}$.

\begin{corollary}[Kernel representation]\label{cor:kernel:representation}
Let
\begin{align*}
  h_{\boldsymbol w}(u_1,v_1,u_2,v_2)
  &=
  w_\tau\sgn\!\left\{
    (u_1-u_2)(v_1-v_2)
  \right\}
  +
  w_{\mathrm{rc}}\sgn\!\left\{
    (u_1+u_2-1)(v_1+v_2-1)
  \right\}\\
  &\quad
  +
  \frac{w_{\mathrm{dc}}}{2}
  \sgn\!\left\{
    (v_1-u_2)(u_1-v_2)
  \right\}
  +
  \frac{w_{\mathrm{dc}}}{2}
  \sgn\!\left\{
    (v_1+u_2-1)(u_1+v_2-1)
  \right\}.
\end{align*}
Then this kernel is bounded and symmetric, and
\begin{align*}
 \kappa_{\boldsymbol w}(C)= \E\left[
    h_{\boldsymbol w}\left(U_1,V_1,U_2,V_2\right)
  \right],
\end{align*}
where $(U_1,V_1)$ and $(U_2,V_2)$ are independent copies of $(U,V)\sim C$.
\end{corollary}

\section{Conclusion}\label{sec:conclusion}
We provide a characterization of all
bivariate measures of concordance of degree at most two based on canonical
polarization.
We show that such measures are characterized by a measure field satisfying several properties including square-group equivariance.
Leveraging this characterization, we provide constructions of measures of concordance of degree two based on copula operators and, as a special case, averaging operators of square-group transforms.

Natural directions for future research include studying the geometry of the class of admissible measure fields.
It would also be of interest to develop analogous complete characterizations for bivariate measures of concordance of degree greater than two and for multivariate measures of concordance of prescribed degree.

\section*{Funding sources}
Takaaki Koike is supported by the Japan Society for the Promotion of Science
(JSPS) KAKENHI grant numbers JP24K00273 and JP26K21178.

\section*{Declaration of competing interests}
The authors declare none.

\section*{Declaration of generative AI and AI-assisted technologies in the manuscript preparation process}
During the preparation of this work, the authors used ChatGPT
(GPT-5.6 Sol; OpenAI) and Claude (Claude Fable 5; Anthropic) to
improve the language, clarity, and readability of the manuscript,
support the refinement and presentation of ideas, and identify
possible errors, gaps, or inconsistencies in mathematical statements
and proofs. The authors independently verified all mathematical
statements and proofs. After using these tools, the authors reviewed
and edited the content as needed and take full responsibility for the
content of the published article.

% \section*{CRediT author statement}
% Conceptualization, T.~Koike; methodology, T.~Koike.; validation, T.~Koike. and H.~Tsunekawa; formal analysis, T.~Koike. and H.~Tsunekawa; investigation, T.~Koike.and H.~Tsunekawa; resources, T.~Koike.; writing---original draft preparation, T.~Koike.; writing---review and editing, T.~Koike. and H.~Tsunekawa; project administration, T.~Koike.; funding acquisition,  T.~Koike.
%All authors have read and agreed to the published version of the manuscript.

\bibliographystyle{apalike}
\bibliography{references}

\newpage
\appendix
\section*{Appendices}

\section{Proofs for Section~\ref{sec:polarization}}\label{sec:proofs}

\subsection{Proof of Proposition~\ref{prop:canonical-polarization}}

We first prove the following lemma.

\begin{lemma}[Quadraticity on copula triangles]
\label{lem:copula-triangle-quadratic}
Let $\kappa:\Cop\to\R$ be a functional of degree at most two.
For any $C_0,C_1,C_2\in\Cop$, there exist constants
$a_0,a_1,a_2,a_{11},a_{12},a_{22}\in\R$ such that
\begin{align*}
  \kappa\bigl((1-s-t)C_0+sC_1+tC_2\bigr)
  &=
  a_0+a_1s+a_2t
  +a_{11}s^2+a_{12}st+a_{22}t^2
\end{align*}
for all $s,t\geq0$ with $s+t\leq1$.
\end{lemma}

\begin{proof}[Proof of Lemma~\ref{lem:copula-triangle-quadratic}]
Let $\Delta_2
  =
  \left\{
    (s,t)\in\R^2:
    s\geq0,\ t\geq0,\ s+t\leq1
  \right\}$ and define $\Psi:\Delta_2\to \Cop$ and $g:\Delta_2\to \R$, respectively, by
\begin{align*}
  \Psi(s,t)
  &=
  (1-s-t)C_0+sC_1+tC_2,\qquad (s,t)\in\Delta_2,\\
  g(s,t)
  &=
  \kappa\bigl(\Psi(s,t)\bigr),
  \qquad (s,t)\in\Delta_2.
\end{align*}
For any $\bm x,\bm y\in\Delta_2$ and $\lambda\in[0,1]$,
the affinity of $\Psi$ gives $\Psi\bigl((1-\lambda)\bm x+\lambda\bm y\bigr)
  =
  (1-\lambda)\Psi(\bm x)+\lambda\Psi(\bm y)$.
Consequently, since $\kappa$ is of degree at most two and $\Psi(\bm x),\Psi(\bm y)\in \Cop$, we have that
\begin{align}
g\bigl((1-\lambda)\bm x+\lambda\bm y\bigr)
  =
  \kappa\bigl(
    (1-\lambda)\Psi(\bm x)+\lambda\Psi(\bm y)
  \bigr),\qquad \lambda \in [0,1],
  \label{eq:g-on-segment}
\end{align}
is a univariate polynomial in $\lambda$ of degree at most two.

Choose $a<b$ and $c<d$ such that $[a,b]\times[c,d]
  \subset\operatorname{int}(\Delta_2)$,
and set $I=(a,b)$ and $J=(c,d)$.
For each fixed $t\in J$, every $s\in I$ can be written as
$(s,t)
  =
  (1-\lambda)(a,t)+\lambda(b,t)$, where $\lambda=(s-a)/(b-a)$.
  Thus, applying \eqref{eq:g-on-segment} with
$\bm x=(a,t)$ and $\bm y=(b,t)$ shows that
$s\mapsto g(s,t)$ is a univariate polynomial of degree at most two on $I$. 
Similarly, for every fixed $s\in I$, the function
$t\mapsto g(s,t)$ is a univariate polynomial of degree at most two
on $J$.

Choose distinct points $u_0,u_1,u_2\in I$ and  $v_0,v_1,v_2\in J$.
For $i,j\in\{0,1,2\}$, define the Lagrange basis polynomials:
\begin{align*}
  L_i(s)
  =
  \prod_{\substack{\ell=0\\\ell\neq i}}^2
  \frac{s-u_\ell}{u_i-u_\ell}\qquad\text{and}\qquad  
  M_j(t)
  =
  \prod_{\substack{\ell=0\\\ell\neq j}}^2
  \frac{t-v_\ell}{v_j-v_\ell},
  \qquad s,t\in\R.
\end{align*}
Define a bivariate polynomial $P:\R^2\to\R$ by
\begin{align*}
  P(s,t)
  &=
  \sum_{i=0}^2\sum_{j=0}^2
  g(u_i,v_j)L_i(s)M_j(t).
\end{align*}
For every $(s,t)\in I\times J$, successive Lagrange interpolation
in the two coordinate variables gives
\begin{align*}
  g(s,t)
  =
  \sum_{i=0}^2L_i(s)g(u_i,t)=
  \sum_{i=0}^2\sum_{j=0}^2
  g(u_i,v_j)L_i(s)M_j(t)
  =
  P(s,t).
\end{align*}
Consequently, we have that
\begin{align}
  g=P
  \qquad\text{on }I\times J.
  \label{eq:g-equals-P-rectangle}
\end{align}
The polynomial $P$ has degree at most two in each variable and hence
has total degree at most four.

Fix $\bm z\in I\times J$. Since $P$ is of degree at most four, we can uniquely write
\begin{align*}
  P(\bm z+\bm h)
  &=
  \sum_{k=0}^4H_k(\bm h),
  \qquad \bm h\in\R^2,
\end{align*}
where each $H_k:\R^2\to\R$, depending on $\bm z$, is a homogeneous polynomial of degree $k$ in the sense that $H_k(\tau\bm h)=\tau^k H_k(\bm h)$ for every $\tau\in \R$ and $\bm h\in \R^2$.

Fix $\bm h\in\R^2$ in addition to $\bm z\in I\times J$. Since $I\times J$ is open and contains $\bm z$,
there exists $\varepsilon>0$ such that $\bm z+\tau\bm h\in I\times J$ for all $|\tau|\leq\varepsilon$.
Applying \eqref{eq:g-on-segment} with $\bm x=\bm z-\varepsilon\bm h$ and $\bm y=\bm z+\varepsilon\bm h$ shows that $\lambda
  \mapsto
  g(\bm z+(2\lambda-1)\varepsilon\bm h)$, $\lambda \in [0,1]$, is a univariate polynomial of degree at most two. 
  By setting $\tau=(2\lambda-1)\varepsilon \in [-\varepsilon,\varepsilon]$, we have that $\tau
  \mapsto
  g(\bm z+\tau\bm h)$, $|\tau|\le \varepsilon$, is a univariate polynomial in $\tau$ of degree at most two.
By \eqref{eq:g-equals-P-rectangle}, the same is true for $P(\bm z+\tau\bm h)
  =
  \sum_{k=0}^4\tau^kH_k(\bm h)$.
This implies that $H_3(\bm h)=H_4(\bm h)=0$.
Since $\bm h\in\R^2$ was arbitrary, we have $H_3\equiv0$ and
$H_4\equiv0$. Hence $P$ has total degree at most two.

It remains to extend \eqref{eq:g-equals-P-rectangle} from $I\times J$ to $\Delta_2$. 
With $\bm z\in I\times J$ still being fixed, we fix $\bm x\in\Delta_2$ and define $g_{\bm x}:[0,1]\to\R$ and  $P_{\bm x}:[0,1]\to\R$ by
\begin{align*}
  g_{\bm x}(\lambda)
  =
  g\bigl((1-\lambda)\bm z+\lambda\bm x\bigr)\qquad\text{and}\qquad
  P_{\bm x}(\lambda)
  =
  P\bigl((1-\lambda)\bm z+\lambda\bm x\bigr),
  \qquad\lambda\in[0,1].
\end{align*}
By \eqref{eq:g-on-segment}, $g_{\bm x}$ is a univariate polynomial
of degree at most two. Since $P$ has total degree at most two,
$P_{\bm x}$ is also a univariate polynomial of degree at most two.

Since $I\times J$ is open, there exists
$\delta>0$ such that $(1-\lambda)\bm z+\lambda\bm x
  \in I\times J$ for all $0\leq\lambda\leq\delta$.
It follows from \eqref{eq:g-equals-P-rectangle} that
$g_{\bm x}=P_{\bm x}$ on $[0,\delta]$. Since these are univariate
polynomials, they must be identical. Therefore, evaluating at $\lambda=1$ gives $g(\bm x)=P(\bm x)$, that is, $g=P$ on $\Delta_2$.
Consequently, the desired statement follows since $P$ has total degree at most two.
\end{proof}

Using this lemma, we prove the main proposition in the following.

\begin{proof}[Proof of Proposition~\ref{prop:canonical-polarization}]
\hspace{0mm}
\begin{enumerate}[label=\textup{(\roman*)},leftmargin=*,widest=ii]
\item 

We first verify joint continuity. Let
$(C_n,D_n)\to(C,D)$ in the product uniform topology, that is, $\|C_n-C\|_\infty\to0$ and $\|D_n-D\|_\infty\to0$.
Then
\begin{align*}
  \left\|
    \frac{C_n+D_n}{2}-\frac{C+D}{2}
  \right\|_\infty
  \leq
  \frac{1}{2}\|C_n-C\|_\infty
  +
  \frac{1}{2}\|D_n-D\|_\infty
  \rightarrow0.
\end{align*}
Therefore, by the continuity of $\kappa$ and the definition of the
canonical polarization, we conclude that $b_\kappa$ is jointly continuous.

Symmetry follows directly from the definition of
$b_\kappa$.

To prove affinity in each argument, fix $C_0,C_1,D\in\Cop$.
Let $\Delta_2$, $\Psi$, and $g$ be as in the proof of
Lemma~\ref{lem:copula-triangle-quadratic}, with $C_2=D$.
Thus, $g=\kappa\circ\Psi$, and Lemma~\ref{lem:copula-triangle-quadratic}
gives constants
$a_0,a_1,a_2,a_{11},a_{12},a_{22}\in\R$ such that
\begin{align}
  g(s,t)
  &=
  a_0+a_1s+a_2t
  +a_{11}s^2+a_{12}st+a_{22}t^2,
  \qquad (s,t)\in\Delta_2.
  \label{eq:local-quadratic-representation}
\end{align}
For $\lambda\in[0,1]$, let $C_\lambda
  =
  (1-\lambda)C_0+\lambda C_1$.
By the definition of $\Psi$, we have that  $C_\lambda
  =
  \Psi(\lambda,0)$, $D
  =
  \Psi(0,1)$ and $(C_\lambda+D)/2=\Psi\left(\lambda/2,1/2
  \right)$.
Consequently, substituting
\eqref{eq:local-quadratic-representation} into the definition of
$b_\kappa$ gives
\begin{align*}
  b_\kappa(C_\lambda,D)
  &=
  2g\left(\frac{\lambda}{2},\frac12\right)
  -\frac{g(\lambda,0)+g(0,1)}{2}\\
  &=
  2\left(
    a_0+\frac{\lambda}{2}a_1+\frac12a_2
    +\frac{\lambda^2}{4}a_{11}
    +\frac{\lambda}{4}a_{12}
    +\frac14a_{22}
  \right)\\
  &\quad
  -\frac12
  \left(
    a_0+\lambda a_1+\lambda^2a_{11}
  \right)
  -\frac12
  \left(
    a_0+a_2+a_{22}
  \right)\\
  &=
  a_0+\frac{a_2}{2}
  +\frac{\lambda}{2}(a_1+a_{12}).
\end{align*}
In particular, the cases $\lambda=0,1$ yield
\begin{align*}
  b_\kappa(C_0,D)
  =
  a_0+\frac{a_2}{2}\qquad\text{and}\qquad
  b_\kappa(C_1,D)
  =
  a_0+\frac{a_2}{2}
  +\frac{a_1+a_{12}}{2}.
\end{align*}
Consequently, it can be directly checked that
$b_\kappa(C_\lambda,D)
  =
  (1-\lambda)b_\kappa(C_0,D)
  +\lambda b_\kappa(C_1,D)$.
Since $C_0,C_1,D\in\Cop$ and $\lambda\in[0,1]$ were arbitrary,
$b_\kappa$ is affine in its first argument. By symmetry, it is also
affine in its second argument.

\item
Taking $D=C$ in the definition of $b_\kappa$ gives
\begin{align*}
  b_\kappa(C,C)
  =
  2\kappa\left(\frac{C+C}{2}\right)
  -\frac{\kappa(C)+\kappa(C)}{2}
  =
  \kappa(C).
\end{align*}

\item
Fix $C,D\in\Cop$ and $t\in[0,1]$.
Set $E_t=
  (1-t)C+tD \in \Cop$.
By (ii), we have that $\kappa(E_t)
  =
  b_\kappa(E_t,E_t)$.
Applying affinity in the first argument and then in the second
argument gives
\begin{align*}
  b_\kappa(E_t,E_t)
  &=
  (1-t)b_\kappa(C,E_t)
  +t b_\kappa(D,E_t)\\
  &=
  (1-t)
  \left\{
    (1-t)b_\kappa(C,C)
    +t b_\kappa(C,D)
  \right\}
  +t
  \left\{
    (1-t)b_\kappa(D,C)
    +t b_\kappa(D,D)
  \right\}\\
  &=
  (1-t)^2b_\kappa(C,C)
  +t(1-t)
  \left\{
    b_\kappa(C,D)+b_\kappa(D,C)
  \right\}
  +t^2b_\kappa(D,D).
\end{align*}
Using (i) and (ii), we
obtain~\eqref{eq:quadratic-segment}.

\item
Let $b:\Cop\times\Cop\to\R$ be any symmetric separately affine
map satisfying $b(C,C)=\kappa(C)$,
$C\in\Cop$.
For $C,D\in\Cop$, separate affinity gives
\begin{align*}
  \kappa\left(\frac{C+D}{2}\right)
  =
  b\left(
    \frac{C+D}{2},
    \frac{C+D}{2}
  \right)=
  \frac14b(C,C)
  +\frac14b(C,D)
  +\frac14b(D,C)
  +\frac14b(D,D).
\end{align*}
Since $b$ is symmetric and its diagonal is $\kappa$, we have that
\begin{align*}
  \kappa\left(\frac{C+D}{2}\right)
  &=
  \frac14\kappa(C)
  +\frac12b(C,D)
  +\frac14\kappa(D).
\end{align*}
Solving for $b(C,D)$ yields
\begin{align*}
  b(C,D)
  =
  2\kappa\left(\frac{C+D}{2}\right)
  -\frac{\kappa(C)+\kappa(D)}{2}=
  b_\kappa(C,D).
\end{align*}
Since $C,D\in\Cop$ were arbitrary, we have $b=b_\kappa$.\qedhere
\end{enumerate}
\end{proof}

\subsection{Proposition~\ref{prop:characterization:b}}

\begin{lemma}[Diagonal versus sectionwise monotonicity]\label{lem:monotonicity-equivalence}
Let $\kappa:\Cop\to\R$ be a continuous functional of degree at most two with canonical polarization $b_\kappa$.  Then the following statements are equivalent:
\begin{enumerate}[label=\textup{(\roman*)},leftmargin=*,widest=ii]
  \item $\kappa$ is $\preceq$-monotone;
  \item for every $C\in\Cop$, the section $D\mapsto b_\kappa(C,D)$ is $\preceq$-monotone.
\end{enumerate}
\end{lemma}

\begin{proof}
Assume first that $\kappa$ is $\preceq$-monotone.  Fix $C,D_1,D_2\in\Cop$ with $D_1\preceq D_2$.  For every $t\in[0,1]$, we have $(1-t)C+tD_1\preceq(1-t)C+tD_2$.
Hence Proposition~\ref{prop:canonical-polarization}~\textup{(iii)} yields
\begin{align*}
  0
  &\leq
  \kappa((1-t)C+tD_2)-\kappa((1-t)C+tD_1)\\
  &=2t(1-t)\{b_\kappa(C,D_2)-b_\kappa(C,D_1)\}
    +t^2\{\kappa(D_2)-\kappa(D_1)\}.
\end{align*}
  Dividing by $t>0$ and letting $t\downarrow0$ leads to $b_\kappa(C,D_1)\leq b_\kappa(C,D_2)$.

Conversely, suppose that every section is $\preceq$-monotone, and let
$C\preceq D$. Symmetry and monotonicity in each argument give
\begin{align*}
  \kappa(D)
  =
  b_\kappa(D,D)
  \geq
  b_\kappa(D,C)
  =
  b_\kappa(C,D)
  \geq
  b_\kappa(C,C)
  =
  \kappa(C).
\end{align*}
Hence $\kappa$ is $\preceq$-monotone.
\end{proof}

\begin{proof}[Proof of Proposition~\ref{prop:characterization:b}]
Suppose first that $\kappa$ is a measure of concordance.
Recall that $\kappa(C^\gamma)
  =
  \eps(\gamma)\kappa(C)$, $C\in\Cop$, $\gamma\in\mathcal D_4$.
Since the action of $\mathcal D_4$ on $\Cop$ is affine, we have that $\{(C+D)/2\}^\gamma
  =
  (C^\gamma+D^\gamma)/2$.
It thus follows that
\begin{align*}
  b_\kappa(C^\gamma,D^\gamma)
  &=
  2\kappa\left(\frac{C^\gamma+D^\gamma}{2}\right)
  -\frac{\kappa(C^\gamma)+\kappa(D^\gamma)}{2}\\
  &=
  \eps(\gamma)
  \left\{
    2\kappa\left(\frac{C+D}{2}\right)
    -\frac{\kappa(C)+\kappa(D)}{2}
  \right\}\\
  &=
  \eps(\gamma)b_\kappa(C,D),
\end{align*}
which proves \ref{cond:B1}. By~\ref{item:monotone} and
Lemma~\ref{lem:monotonicity-equivalence}, Condition \ref{cond:B2}
holds. Finally, we have from~\ref{item:M} that $b_\kappa(M,M)
  =
  \kappa(M)
  =
  1$, which is \ref{cond:B3}.

Conversely, suppose that \ref{cond:B1}--\ref{cond:B3} hold.
Using the diagonal identity $b_\kappa(C,C)=\kappa(C)$ in
\ref{cond:B1}, we obtain
\begin{align*}
  \kappa(C^\gamma)
  =
  b_\kappa(C^\gamma,C^\gamma)=
  \eps(\gamma)b_\kappa(C,C)
  =
  \eps(\gamma)\kappa(C),
  \qquad C\in\Cop,\quad \gamma\in\mathcal D_4.
\end{align*}
Taking $\gamma=\pi$ gives~\ref{item:pi:sym}, whereas taking $\gamma=\sigma_1$ gives~\ref{item:sigma:sym}.
Condition \ref{cond:B2} and
Lemma~\ref{lem:monotonicity-equivalence} yield~\ref{item:monotone}. Moreover, \ref{cond:B3} leads to $\kappa(M)
  =
  b_\kappa(M,M)
  =
  1$, which is~\ref{item:M}. The assumed continuity of $\kappa$ with
respect to the uniform topology is precisely axiom~\ref{item:conv}.
Thus $\kappa$ is a measure of concordance.
\end{proof}

\subsection{Proposition~\ref{prop:canonical-polarization-bijection}}

\begin{proof}[Proof of Proposition~\ref{prop:canonical-polarization-bijection}]

It is straightforward to check from Propositions~\ref{prop:canonical-polarization} and~\ref{prop:characterization:b} that $\mathsf P$ maps $\mathcal K_d$ into $\mathcal B_d$.

To show the converse, let $b\in\mathcal B_d$ and define $\kappa_b:\Cop\to\R$ by 
$\kappa_b(C)
  =
  b(C,C)$, $C\in\Cop$.
Then continuity of this functional follows from the joint continuity of $b$.
Since $b \in \mathcal B_d$, we have that $\kappa_b$ is of degree exactly $d$.
In addition, Proposition~\ref{prop:canonical-polarization}~(iv) yields $b=b_{\kappa_b}$.
Since the map $b$ satisfies \ref{cond:B1}--\ref{cond:B3}, it follows from Proposition~\ref{prop:characterization:b}  that $\kappa_b$ is a
measure of concordance, that is, $\mathsf D(b)
  =
  \kappa_b
  \in\mathcal K_d$.

Note that $(\mathsf D\circ\mathsf P)(\kappa)(C)
  =
  b_\kappa(C,C)
  =
  \kappa(C)$ and $
 (\mathsf P\circ\mathsf D)(b)
  =
  b_{\kappa_b}
  =
  b$.
Consequently, $\mathsf P$ and $\mathsf D$ are inverse bijections between
$\mathcal K_d$ and $\mathcal B_d$ for each $d\in\{1,2\}$, and hence also
between $\mathcal K_{1,2}$ and $\mathcal B_{1,2}$.
\end{proof}

\section{Proofs for Section~\ref{sec:characterization:field}}

\subsection{Integral representation for affine functionals}

We shall use the following integral representation of continuous
affine functionals. Related representations for degree-one measures
of concordance can be found in \citet{edwards2005} and
\citet{edwardstaylor2009}. 
Since this formulation is
not stated explicitly in either reference, we provide a self-contained proof
in Section~\ref{sec:proof:section:representation}.

\begin{theorem}[Integral representation of monotone affine functionals]
\label{thm:affine-section-representation}
Let $\Phi:\Cop\to\R$ be a continuous affine functional. Then $\Phi$ is
$\preceq$-monotone if and only if there exists 
$\nu_\Phi\in\Meas_w^+$ such that
\begin{align}
  \Phi(D)
  &=
  \Phi(\Pi)
  +\int_{(0,1)^2}(D-\Pi)\,\dd\nu_\Phi,
  \qquad D\in\Cop.
  \label{eq:affine-section-representation}
\end{align}
The representing measure $\nu_\Phi$ is unique and satisfies
\begin{align}
  \int_{(0,1)^2}w\,\dd\nu_\Phi
  &=
  \Phi(M)-\Phi(W).
  \label{eq:section-weighted-mass}
\end{align}
\end{theorem}

\subsection{Proof of Theorem~\ref{thm:affine-section-representation}}
\label{sec:proof:section:representation}

Sufficiency part is straightforward to prove by the construction~\eqref{eq:affine-section-representation}.
Moreover,~\eqref{eq:section-weighted-mass} follows from~\eqref{eq:affine-section-representation} since
\begin{align*}
  \Phi\left(M\right)-\Phi\left(W\right)
  &=
  \int_{(0,1)^2}
  \left(M-W\right)\,\dd\nu_\Phi
  =
  \int_{(0,1)^2} w\,\dd\nu_\Phi.
\end{align*}
Uniqueness of $\nu_\Phi$ is also a consequence from~\eqref{eq:affine-section-representation}. 
Indeed, if $\nu_\Phi,\widetilde\nu_\Phi\in\Meas_w^+$ both satisfy
\eqref{eq:affine-section-representation}, then
\begin{align*}
  \int_{(0,1)^2}(D-\Pi)\,\dd\nu_\Phi
  &=
  \int_{(0,1)^2}(D-\Pi)\,\dd\widetilde\nu_\Phi,
  \qquad D\in\Cop.
\end{align*}
Since both measures are regular, \citet[Lemma~0.4]{edwards2005} yields
$\nu_\Phi=\widetilde\nu_\Phi$.

We prove the necessity part of Theorem~\ref{thm:affine-section-representation} with the approximation techniques used in~\citet{edwardstaylor2009}. 
We consider the checkerboard approximation of copulas over the grid $\mathcal V_n=\{0,h_n,2h_n,\dots,1\}^2$, where $m_n=2^n$ and $h_n=1/m_n$, $n\in \mathbb{N}$.
For $i,j\in\{0,1,\dots,m_n-1\}$ and a bivariate $\R$-valued function $F$, write the increment of $F$ by
\begin{align*}
  \Delta_{ij,n}F
  =
  F\left(\frac{i+1}{m_n},\frac{j+1}{m_n}\right)
  -
  F\left(\frac{i}{m_n},\frac{j+1}{m_n}\right)
  -
  F\left(\frac{i+1}{m_n},\frac{j}{m_n}\right)
  +
  F\left(\frac{i}{m_n},\frac{j}{m_n}\right).
\end{align*}
Let $S$ be a subcopula with domain $\mathcal V_n$; see~\citet[][Section~2.3.1]{durantesempi2016} for the definition.
We then denote by $C_S$ the \emph{checkerboard copula} associated with $S$ and the grid $\mathcal V_n$.
As in~\citet[][Theorem~4.1.3]{durantesempi2016}, for every $(u,v)\in[i\,h_n,(i+1)\,h_n]\times [j\,h_n,(j+1)\,h_n]$ with $i,j\in\{0,1,\dots,m_n-1\}$, we have
\begin{align}\label{eq:checkerboard:definition}
C_S(u,v)=S\left(i\,h_{n},j\,h_n\right)+\sum_{(k,\ell)\in\mathcal J_{ij}}(\Delta_{k\ell,n}S)\,f_{k,n}(u)f_{\ell,n}(v),\end{align}
where $\mathcal J_{ij}=\{(k,\ell)\in \{0,\dots,i\}\times \{0,\dots,j\}:k=i\text{ or }\ell=j\}$ and $f_{k,n}:[0,1]\to[0,1]$, $k\in\{0,\dots,m_n-1\}$, is defined by
\begin{align*}
f_{k,n}(x)=\begin{cases}
    0, & x<k\,h_n,\\
    (x-k\,h_n)/h_n, & k\,h_n\le x  \le (k+1)\,h_n,\\
    1, & (k+1)\,h_n<x.\\
\end{cases}
\end{align*}
It is shown in~\citet[Theorem~4.1.2]{durantesempi2016} that $C_S$ is indeed a copula. 
Note that $C_S$ agrees with $S$ on $\mathcal V_n$.
For every $(u,v)\in[0,1]^2$, the value
$C_S(u,v)$ can be obtained by bilinear interpolation.
To see this, fix
$i,j\in\left\{0,\ldots,m_n-1\right\}$ and $(u,v)\in[i\,h_n,(i+1)\,h_n]\times [j\,h_n,(j+1)\,h_n]$.
Set $\alpha
  =
  m_n\left(u-i/m_n\right)
  =
  m_n u-i\in[0,1]$ and $\beta
  =
  m_n\left(v-j/m_n\right)
  =
  m_n v-j\in[0,1]$.
With any empty sum understood to be zero,
\eqref{eq:checkerboard:definition} yields
\begin{align}
  C_S(u,v)
  &=
  S\left(\frac{i}{m_n},\frac{j}{m_n}\right)
  +
  \alpha\beta\,\Delta_{ij,n}S
  +
  \alpha
  \sum_{\ell=0}^{j-1}\Delta_{i\ell,n}S
  +
  \beta
  \sum_{k=0}^{i-1}\Delta_{kj,n}S
  \notag\\
  &=
  S\left(\frac{i}{m_n},\frac{j}{m_n}\right)
  +
  \alpha\beta\,\Delta_{ij,n}S
  \notag\\
  &\quad+
  \alpha
  \left\{
    S\left(\frac{i+1}{m_n},\frac{j}{m_n}\right)
    -
    S\left(\frac{i}{m_n},\frac{j}{m_n}\right)
  \right\}
 +
  \beta
  \left\{
    S\left(\frac{i}{m_n},\frac{j+1}{m_n}\right)
    -
    S\left(\frac{i}{m_n},\frac{j}{m_n}\right)
  \right\}
  \notag\\
  &=
  (1-\alpha)(1-\beta)\,
  S\left(\frac{i}{m_n},\frac{j}{m_n}\right)
  +
  \alpha(1-\beta)\,
  S\left(\frac{i+1}{m_n},\frac{j}{m_n}\right)
  \notag\\
  &\quad+
  (1-\alpha)\beta\,
  S\left(\frac{i}{m_n},\frac{j+1}{m_n}\right)
  +
  \alpha\beta\,
  S\left(\frac{i+1}{m_n},\frac{j+1}{m_n}\right),
  \label{eq:checkerboard-bilinear-form}
\end{align}
where the second equality follows by telescoping and the groundedness of
$S$.

For a copula $C\in \Cop$, its restriction  $C|_{\mathcal V_n}$ is a subcopula.
The \emph{checkerboard approximation} of $C$, denoted by $\Xi_n(C)$, is then defined as the checkerboard copula associated with $C|_{\mathcal V_n}$.
We also write $\Cop^\#_n=\{\Xi_n(C):C\in\Cop\}\subset \Cop$.
For any $C\in \Cop$, its checkerboard approximation $\Xi_n(C)$ uniformly converges to $C$~\citep[][Theorem 4.1.5]{durantesempi2016}.

\begin{lemma}[Finite checkerboard representation]
\label{lem:finite-checkerboard}
Let $\Phi:\Cop\to\R$ be affine and $\preceq$-monotone. Then, for every
$n\geq1$, there uniquely exists $\bm a_n(\Phi)=(a_{ij,n}(\Phi))_{i,j\in\{1,\dots,m_n-1\}}$, $a_{ij,n}(\Phi)\geq0$, such that
\begin{align}
  \Phi(D)-\Phi(\Pi)
  &=
  \sum_{i,j=1}^{m_n-1}
  a_{ij,n}(\Phi)
  \left\{
    D\left(\frac{i}{m_{n}},\frac{j}{m_{n}}\right)-\frac{i}{m_{n}}\frac{j}{m_{n}}
  \right\}\qquad \text{for every $D\in\Cop^\#_n$.}
  \label{eq:finite-checkerboard-rep}
\end{align}
\end{lemma}

\begin{proof}
Fix $n\geq1$ and $i,j\in\{1,\dots,m_n-1\}$. Define $L_{i,n}:[0,1]\to [0,1]$ by
\begin{align*}
  L_{i,n}(u)
  =
  \max\left\{
    1-m_n\left|u-i\,h_n\right|,
    0
  \right\},
  \qquad
  u\in[0,1].
\end{align*}
Then $L_{i,n}$ is continuous, is affine on each grid interval
$[k\,h_n,(k+1)\,h_n]$, $k \in \{0,\dots,m_n-1\}$, and satisfies $L_{i,n}(k\,h_n)=\id_{\{k=i\}}$ for $k \in \{1,\dots,m_n-1\}$.
Define the two functions $C_{ij,n}^{+}=
  \Pi+\varepsilon_nH_{ij,n}$ and $C_{ij,n}^{-}
  =
  \Pi-\varepsilon_nH_{ij,n}$, where $H_{ij,n}(u,v)=L_{i,n}(u)L_{j,n}(v)$ for
$(u,v)\in[0,1]^2$ and $\varepsilon_n=h_n^2/2$. 

We will show that $C_{ij,n}^{+},C_{ij,n}^{-}\in\Cop_n^\#$.
First, their restrictions to $\mathcal V_n$ are subcopulas.
Indeed, $H_{ij,n}(u,v)=0$ whenever either $u$ or $v$ belongs to
$\{0,1\}$, which yields the boundary conditions.
For $k,\ell\in\{0,\dots,m_n-1\}$, we have
$\Delta_{k\ell,n}\Pi=h_n^2$ and \begin{align*}
\Delta_{k\ell,n}H_{ij,n}=\left\{L_{i,n}\left(\frac{k+1}{m_n}\right)-L_{i,n}\left(\frac{k}{m_n}\right)\right\}
\left\{L_{j,n}\left(\frac{\ell+1}{m_n}\right)-L_{j,n}\left(\frac{\ell}{m_n}\right)\right\}\in\{-1,0,1\}.
  \end{align*}
Therefore, $\Delta_{k\ell,n}C_{ij,n}^{\pm}
  =\Delta_{k\ell,n}\Pi\pm \varepsilon_n\Delta_{k\ell,n}H_{ij,n}
  \geq h_n^2-\varepsilon_n=h_n^2/2>0$, showing $2$-increasingness on $\mathcal V_n$.
Next, we prove bilinearity~\eqref{eq:checkerboard-bilinear-form} on each cell.
Fix $k,\ell\in\{0,\ldots,m_n-1\}$ and $(u,v)
  \in
  \left[k\,h_n,(k+1)\,h_n\right]
  \times
  \left[\ell\,h_n,(\ell+1)\,h_n\right]$,
and set $\alpha=m_nu-k$ and $\beta=m_n v - \ell$.
Then $u=(1-\alpha)k\,h_n
  +
  \alpha(k+1)\,h_n$ and $v
  =
  (1-\beta)\ell\,h_n
  +
  \beta(\ell+1)\,h_n$.
By affinity of $L_{i,n}$ and $L_{j,n}$, it follows that
\begin{align*}
  C_{ij,n}^{\pm}(u,v)
  &=
  (1-\alpha)(1-\beta)
  C_{ij,n}^{\pm}\left(
    \frac{k}{m_n},
    \frac{\ell}{m_n}
  \right)
  +
  \alpha(1-\beta)
  C_{ij,n}^{\pm}\left(
    \frac{k+1}{m_n},
    \frac{\ell}{m_n}
  \right)
  \\
  &\quad+
  (1-\alpha)\beta
  C_{ij,n}^{\pm}\left(
    \frac{k}{m_n},
    \frac{\ell+1}{m_n}
  \right)
 +
  \alpha\beta
  C_{ij,n}^{\pm}\left(
    \frac{k+1}{m_n},
    \frac{\ell+1}{m_n}
  \right).
\end{align*}
Consequently, $C_{ij,n}^{+}$ and $C_{ij,n}^{-}$ are copulas and, in particular, $C_{ij,n}^{\pm}=\Xi_n\left(C_{ij,n}^{\pm}\right)\in\Cop_n^\#$.

Define
\begin{align*}
  a_{ij,n}(\Phi)
  =
  \frac{
    \Phi(C_{ij,n}^{+})-\Phi(\Pi)
  }{\varepsilon_n}.
\end{align*}
Since $C_{ij,n}^{+}-\Pi=\varepsilon_nH_{ij,n}\geq0$, we have
$\Pi\preceq C_{ij,n}^{+}$. The $\preceq$-monotonicity of $\Phi$
therefore gives $a_{ij,n}(\Phi)\geq0$. Furthermore, we have
$\Pi=(C_{ij,n}^{+}+C_{ij,n}^{-})/2$, and thus the affinity of $\Phi$
yields $\Phi(C_{ij,n}^{-})-\Phi(\Pi)
  =
  -\varepsilon_na_{ij,n}(\Phi)$.
Now let $D\in\Cop_n^\#$ and define
\begin{align*}
  d_{ij,n}
  &=
  D\left(\frac{i}{m_n},\frac{j}{m_n}\right)
  -
  \frac{i}{m_n}\frac{j}{m_n},
  \qquad
i,j\in \{1,\dots,m_n-1\}.
\end{align*}
Then $D-\Pi$ and
$\sum_{i,j=1}^{m_n-1}d_{ij,n}H_{ij,n}$ agree on $\mathcal V_n$.
Moreover, both $D-\Pi$ and
$\sum_{i,j=1}^{m_n-1}d_{ij,n}H_{ij,n}$ are bilinear on every grid
cell. 
Consequently, we have $D-\Pi
  =
  \sum_{i,j=1}^{m_n-1}
  d_{ij,n}H_{ij,n}$ on $[0,1]^2$.

Set $S_n=\sum_{i,j=1}^{m_n-1}|d_{ij,n}|$. If $S_n=0$, then $D=\Pi$ and thus~\eqref{eq:finite-checkerboard-rep} is
immediate. 
Suppose that $S_n>0$, and let $\theta_n
  =
  \min\left\{
    1,
    \varepsilon_n/S_n
  \right\}$.
Write $d_{ij,n}=d_{ij,n}^{+}-d_{ij,n}^{-}$, where $d_{ij,n}^{+}
  =
  \max\{d_{ij,n},0\}$ and $d_{ij,n}^{-}
  =
  \max\{-d_{ij,n},0\}$,
and define $\lambda_{0,n}
  =
  1-\theta_nS_n/\varepsilon_n\ge 0$ and $\lambda_{ij,n}^{\pm}
  =
  \theta_nd_{ij,n}^{\pm}/\varepsilon_n\ge 0$ so that $\lambda_{0,n}
  +
  \sum_{i,j=1}^{m_n-1}
  \left(
    \lambda_{ij,n}^{+}
    +
    \lambda_{ij,n}^{-}
  \right)
  =
  1.$
 We then obtain
\begin{align*}
  \lambda_{0,n}\Pi
  +
  \sum_{i,j=1}^{m_n-1}
  \lambda_{ij,n}^{+}C_{ij,n}^{+}
  +
  \sum_{i,j=1}^{m_n-1}
  \lambda_{ij,n}^{-}C_{ij,n}^{-}
  &=
  \Pi
  +
  \varepsilon_n
  \sum_{i,j=1}^{m_n-1}
  \left(
    \lambda_{ij,n}^{+}
    -
    \lambda_{ij,n}^{-}
  \right)H_{ij,n}\\
  &=
  \Pi
  +
  \theta_n
  \sum_{i,j=1}^{m_n-1}
  d_{ij,n}H_{ij,n}\\
  &=
  (1-\theta_n)\Pi+\theta_nD.
\end{align*}
The affinity of $\Phi$ yields
\begin{align*}
  \theta_n
  \left\{
    \Phi(D)-\Phi(\Pi)
  \right\}
  &=
  \sum_{i,j=1}^{m_n-1}
  \lambda_{ij,n}^{+}
  \left\{
    \Phi(C_{ij,n}^{+})-\Phi(\Pi)
  \right\}+
  \sum_{i,j=1}^{m_n-1}
  \lambda_{ij,n}^{-}
  \left\{
    \Phi(C_{ij,n}^{-})-\Phi(\Pi)
  \right\}\\
  &=
  \theta_n
  \sum_{i,j=1}^{m_n-1}
  a_{ij,n}(\Phi)
  \left(
    d_{ij,n}^{+}-d_{ij,n}^{-}
  \right)\\
  &=
  \theta_n
  \sum_{i,j=1}^{m_n-1}
  a_{ij,n}(\Phi)d_{ij,n}.
\end{align*}
Since $\theta_n>0$, we obtain the desired representation~\eqref{eq:finite-checkerboard-rep}.

It remains to prove uniqueness of the coefficients $(a_{ij,n}(\Phi))_{i,j\in\{1,\dots,m_n-1\}}$. Suppose that
$(b_{ij,n})_{i,j\in\{1,\dots, m_n-1\}}$ yield the same representation. Fix
$i,j\in\{1,\dots,m_n-1\}$. 
By~\eqref{eq:finite-checkerboard-rep}, we have $\Phi(C_{ij,n}^{+})-\Phi(\Pi)
  =
  \varepsilon_n\,b_{ij,n}$ and $\Phi(C_{ij,n}^{+})-\Phi(\Pi)
  =
  \varepsilon_n\,a_{ij,n}(\Phi)$.
   Hence
$b_{ij,n}=a_{ij,n}(\Phi)$. 
\end{proof}
For each $n\geq1$, define the finite non-negative Borel measures
$\nu_n$ and $\eta_n$ on $[0,1]^2$ by
\begin{align*}
  \nu_n
  =
  \sum_{i,j=1}^{m_n-1}
  a_{ij,n}\left(\Phi\right)\,
  \delta_{\left(i\,h_n,j\,h_n\right)}\qquad\text{and}\qquad
  \eta_n
  =
  \sum_{i,j=1}^{m_n-1}
  a_{ij,n}\left(\Phi\right)
  w\left(i\,h_n,j\,h_n\right)
  \delta_{\left(i\,h_n,j\,h_n\right)}.
\end{align*}
Note that $\eta_n(A)
  =
  \int_A w\,\dd\nu_n$ for every
$A\in\mathcal B\left([0,1]^2\right)$.
Applying Lemma~\ref{lem:finite-checkerboard} to $\Xi_n(M)$ and $\Xi_n(W)$ and using continuity of $\Phi$, we have that, as $n\to \infty$,
\begin{align*}
  \eta_n\left(\left[0,1\right]^2\right)
  &=
  \int_{\left(0,1\right)^2}w\,\dd\nu_n
  =
  \Phi\left(\Xi_n\left(M\right)\right)
  -
  \Phi\left(\Xi_n\left(W\right)\right)\rightarrow
  \Phi\left(M\right)-\Phi\left(W\right).
\end{align*}
Since $[0,1]^2$ is compact, the family $\{\eta_n:n\geq1\}$ is uniformly
tight.
Together with the uniform boundedness in variation $\sup_{n\geq1}\|\eta_n\|_{\mathrm{TV}}
  \leq
  \sup_{n\geq1}\eta_n([0,1]^2)<\infty$, it follows from~\citet[Theorem~2.3.4]{bogachev2018} that there exists a subsequence
$(\eta_{n_k})_{k\geq1}$ weakly converging to a finite non-negative Borel
measure $\eta$ on $[0,1]^2$.
Note that $\eta\left([0,1]^2\right)
  =
  \lim_{k\to\infty}
  \eta_{n_k}\left([0,1]^2\right)
  =
  \Phi\left(M\right)-\Phi\left(W\right)<\infty$.

For $k\geq2$, set
$R_k=[h_k,1-h_k]^2$.
Following \citet[Section~3.5]{edwardstaylor2009}, the following map $\Box_k$ on $\Cop$ defines a copula: 
\begin{align*}
  \Box_k(C)(u,v)
  &=
  \begin{dcases}
    \left(1-2h_k\right)^2
    C\left(
    \frac{u-h_k}{1-2h_k},
    \frac{v-h_k}{1-2h_k}
  \right)
    +h_k(u+v)-h_k^2,
    &(u,v)\in R_k,\\[2pt]
    uv,
    &(u,v)\in[0,1]^2\setminus R_k.
  \end{dcases}
\end{align*}
The following properties follow from~\citet[Lemma~16]{edwardstaylor2009}.
First, $\Box_k(C)$ converges pointwise to $C$ for every $C\in \Cop$.
Second, $w_k 
  =
  \Box_k(M)-\Box_k(W)$ is continuous and vanishes on $[0,1]^2\setminus R_k$.
  Moreover, since $W
  \preceq
  \Box_k(W)
  \preceq
  \Pi
  \preceq
  \Box_k(M)
  \preceq
  M$, we have $0\leq w_k\leq w$.

\begin{lemma}[Boundary tightness]
\label{lem:checkerboard-boundary-tightness}
Let $\Phi:\Cop\to\R$ be continuous, affine, and
$\preceq$-monotone, and set
$S_\delta:=\{(u,v)\in[0,1]^2:w(u,v)<\delta\}$ for $\delta>0$. Then
\begin{align*}
  \lim_{\delta\downarrow0}
  \limsup_{n\to\infty}
  \eta_n\left(S_\delta\right)
  &=
  0.
  % \label{eq:boundary-tightness}
\end{align*}
\end{lemma}

\begin{proof}
For each $k\geq2$, set $\underline{\ell}_k
  =
  \min_{(u,v)\in R_k}w(u,v)>0$. 
Fix $\delta\in(0,\underline{\ell}_k)$ so that $S_\delta\cap R_k=\emptyset$, and hence $w_k=0$ on
$S_\delta$. Therefore, using the non-negativity of the coefficients $(a_{ij,n}(\Phi))$ in Lemma~\ref{lem:finite-checkerboard} and the inequality
$w-w_k\geq0$, we obtain
\begin{align*}
  \eta_n(S_\delta)
  &=
  \sum_{\substack{1\leq i,j\leq m_n-1\\
    (i\,h_n,j\,h_n)\in S_\delta}}
  a_{ij,n}(\Phi)
  \left(w-w_k\right)
  \left(i\,h_n,j\,h_n\right)\\
& \leq
  \sum_{i,j=1}^{m_n-1}
  a_{ij,n}(\Phi)
  \left(w-w_k\right)
  \left(i\,h_n,j\,h_n\right)\\
  &= \Phi\left(\Xi_n(M)\right)
  -
  \Phi\left(\Xi_n(W)\right)
  -
  \Phi\left(\Xi_n\left(\Box_k(M)\right)\right)
  +
  \Phi\left(\Xi_n\left(\Box_k(W)\right)\right).
\end{align*}
For this fixed $k$, letting $n\to\infty$ yields, for every $0<\delta<\underline{\ell}_k$,
\begin{align}\label{eq:eta:bound}
  \limsup_{n\to\infty}\eta_n(S_\delta)
  &\leq
  \Phi(M)
  -
  \Phi(W)
  -
  \Phi\left(\Box_k(M)\right)
  +
  \Phi\left(\Box_k(W)\right)
=:
  \overline{\eta}_k(\Phi).
\end{align}
Since $S_\delta$ is a non-decreasing set in $\delta$, the function
$\delta\mapsto\limsup_{n\to\infty}\eta_n(S_\delta)$ is non-decreasing and bounded from above. 
Hence its limit as $\delta\downarrow0$ exists, and~\eqref{eq:eta:bound} gives $ 0
  \leq
  \lim_{\delta\downarrow0}
  \limsup_{n\to\infty}\eta_n(S_\delta)
  \leq
  \overline{\eta}_k(\Phi)$ for every $k\geq2$.
  Since $\overline{\eta}_k(\Phi)\rightarrow 0$ by continuity of $\Phi$, we obtain the desired result.
\end{proof}

We are now ready to prove Theorem~\ref{thm:affine-section-representation}.

\begin{proof}[Proof of Theorem~\ref{thm:affine-section-representation} (Necessity)]
Assume that $\Phi$ is continuous, affine, and
$\preceq$-monotone.
Choose a subsequence
$n_k\uparrow\infty$ such that $\eta_{n_k}$ converges weakly to a
finite non-negative Borel measure $\eta$ on $[0,1]^2$.
We first show that $\eta$ is concentrated on $(0,1)^2$, that is, $\eta(\partial [0,1]^2)=0$.
Since $w$ is continuous, each $S_\delta$ is open in $[0,1]^2$.
Therefore, the portmanteau theorem~\citep[][Theorem~2.1]{billingsley1999convergence} yields $\eta\left(S_\delta\right)
  \leq
  \liminf_{k\to\infty}
  \eta_{n_k}\left(S_\delta\right)$.
  Since $S_\delta\downarrow \{w=0\}=\partial [0,1]^2$ as $\delta\downarrow0$, continuity of $\eta$ from above implies $\eta\left(\partial [0,1]^2\right)
  =
  \lim_{\delta\downarrow0}
  \eta\left(S_\delta\right)$.
  Combining them with Lemma~\ref{lem:checkerboard-boundary-tightness}, we have that
  \begin{align*}
      \eta\left(\partial [0,1]^2\right)
  =
  \lim_{\delta\downarrow0}
  \eta\left(S_\delta\right)\le   
  \lim_{\delta\downarrow0}
  \liminf_{k\to\infty}
  \eta_{n_k}\left(S_\delta\right)
  \le 
  \lim_{\delta\downarrow0}
  \limsup_{n\to\infty}
  \eta_n\left(S_\delta\right)
  =
  0.
  \end{align*}
  Since $\eta$ is a non-negative measure, we have $ \eta\left(\partial [0,1]^2\right)=0$.

For every Borel set $A\subset (0,1)^2$, define $\nu(A)
  =
  \int_A w^{-1}\,\dd\eta$.
Since $w^{-1}$ is positive and Borel measurable on $(0,1)^2$, this defines a non-negative Borel measure on $(0,1)^2$. Furthermore,
\begin{align*}
  \int_{(0,1)^2}w\,\dd\nu
  &=
  \eta\left((0,1)^2\right)
  =
  \eta\left([0,1]^2\right)
  =
  \Phi(M)-\Phi(W)
  <
  \infty.
\end{align*}
Thus $\nu\in\Meas_w^+$, and
\eqref{eq:section-weighted-mass} holds with $\nu_\Phi=\nu$.
  
We next prove \eqref{eq:affine-section-representation} for
checkerboard copulas. Let $D\in\Cop_{n_D}^{\#}$ for some $n_D\geq1$,
and define $r_D=(D-\Pi)/w$ on $(0,1)^2$. The function $r_D$ is continuous on $(0,1)^2$.
Moreover, $\lvert r_D\rvert\leq1$ because $\lvert D-\Pi\rvert
  \leq
  \max\left\{M-\Pi,\Pi-W\right\}
  \leq
  w$.

Since $n_k\to\infty$, there exists $k_D\geq1$ such that
$n_k\geq n_D$ for every $k\geq k_D$.
Note that $D\in\Cop_{n_k}^\#$ for every $k\geq k_D$.
Hence, Lemma~\ref{lem:finite-checkerboard} yields
\begin{align}\label{eq:Phi:diff:r:D}
  \Phi(D)-\Phi(\Pi)
  &=
  \int_{(0,1)^2}(D-\Pi)\,\dd\nu_{n_k}
  =
  \int_{(0,1)^2}r_D\,\dd\eta_{n_k}
  \qquad\text{for every $k\geq k_D$.}
\end{align}

For $\delta>0$, define $\chi_\delta:[0,1]^2\to[0,1]$ by
\begin{align*}
  \chi_\delta(u,v)
  =
  \begin{cases}
    0,
    & \text{if }w(u,v)\leq\delta/2,\\
    2w(u,v)/\delta-1,
    & \text{if }\delta/2<w(u,v)<\delta,\\
    1,
    & \text{if }\delta\leq w(u,v).
  \end{cases}
\end{align*}
Then $\chi_\delta=0$ on $S_{\delta/2}$ and
$\chi_\delta=1$ on $[0,1]^2\setminus S_\delta$. Since
$S_{\delta/2}$ is an open neighborhood of $\partial[0,1]^2$ in
$[0,1]^2$ and $\chi_\delta r_D=0$ on
$S_{\delta/2}\cap(0,1)^2$, setting $\chi_\delta r_D$ equal to zero on
$\partial[0,1]^2$ defines a continuous extension to $[0,1]^2$. We
denote this extension by $(\chi_\delta r_D)^\ast$.
By the weak convergence
of $\eta_{n_k}$ to $\eta$, we have $\int_{[0,1]^2}
  (\chi_\delta r_D)^\ast\,\dd\eta_{n_k}
\rightarrow
  \int_{[0,1]^2}
  (\chi_\delta r_D)^\ast\,\dd\eta$.
Since $\lvert r_D\rvert\leq1$, $0\leq\chi_\delta\leq1$, and
$1-\chi_\delta$ vanishes outside $S_\delta$, it follows
from~\eqref{eq:Phi:diff:r:D} that
\begin{align*}
  \left|
    \Phi\left(D\right)-\Phi\left(\Pi\right)
    -
    \int_{[0,1]^2}(\chi_\delta r_D)^\ast\,\dd\eta
  \right|
  &=
  \lim_{k\to\infty}
  \left|
    \int_{(0,1)^2}r_D\,\dd\eta_{n_k}
    -
    \int_{[0,1]^2}(\chi_\delta r_D)^\ast\,\dd\eta_{n_k}
  \right|
  \\
  &=
  \lim_{k\to\infty}
  \left|
    \int_{(0,1)^2}
    \left(1-\chi_\delta\right)r_D\,\dd\eta_{n_k}
  \right|
  \\
  &\leq
  \limsup_{k\to\infty}
  \eta_{n_k}\left(S_\delta\right)
  \\
  &\leq
  \limsup_{n\to\infty}
  \eta_n\left(S_\delta\right).
\end{align*}
The right-hand side tends to zero as $\delta\downarrow0$ by
Lemma~\ref{lem:checkerboard-boundary-tightness}.
For the left-hand side, as $\delta\downarrow0$,
$(\chi_\delta r_D)^\ast$ converges pointwise on $[0,1]^2$ to $r_D^\ast$, the
function that equals $r_D$ on $(0,1)^2$ and zero on
$\partial[0,1]^2$.
Moreover, we have $\lvert(\chi_\delta r_D)^\ast\rvert\leq1$. Since $\eta$ is finite and
concentrated on $(0,1)^2$, the dominated convergence theorem gives
$\int_{[0,1]^2}
  (\chi_\delta r_D)^\ast\,\dd\eta
\rightarrow
  \int_{[0,1]^2}r_D^\ast\,\dd\eta=  \int_{(0,1)^2}r_D\,\dd\eta$.
Hence
\begin{align}\label{eq:desired:checkerboard}
  \Phi\left(D\right)-\Phi\left(\Pi\right)
  &=
  \int_{(0,1)^2}r_D\,\dd\eta
  =
  \int_{(0,1)^2}
  \left(D-\Pi\right)\,\dd\nu,
  \qquad
  D\in\bigcup_{n\geq1}\Cop_n^\#.
\end{align}

Finally, let $C\in\Cop$ be arbitrary. Since
$\Xi_n(C)\in\Cop_n^\#$, it follows from~\eqref{eq:desired:checkerboard} that
$\Phi(\Xi_n(C))-\Phi(\Pi)
=
\int_{(0,1)^2}(\Xi_n(C)-\Pi)\,\dd\nu$
for every $n\geq1$.
On $(0,1)^2$, set
$r_n=(\Xi_n(C)-\Pi)/w$ and $r=(C-\Pi)/w$. Then
$r_n\to r$ pointwise, while $\lvert r_n\rvert\leq1$ and
$\lvert r\rvert\leq1$.
Therefore, the dominated convergence theorem for the
finite measure $\eta$ yields
\begin{align}\label{eq:Xi:desired}
  \int_{(0,1)^2}
  \left(\Xi_n\left(C\right)-\Pi\right)\,\dd\nu
  =
  \int_{(0,1)^2}r_n\,\dd\eta
  \rightarrow
  \int_{(0,1)^2}r\,\dd\eta
  =
  \int_{(0,1)^2}
  \left(C-\Pi\right)\,\dd\nu.
\end{align}
Moreover, $\Xi_n(C)\to C$ uniformly, and hence the continuity of
$\Phi$ gives $\Phi(\Xi_n(C))\to\Phi(C)$. 
Combining this with~\eqref{eq:Xi:desired} leads to
\eqref{eq:affine-section-representation}, which concludes the proof.
\end{proof}

\subsection{Necessity part of Theorem~\ref{thm:measure:field:characterization}}

\begin{proof}[Proof of Theorem~\ref{thm:measure:field:characterization}~\ref{item:characterization:if}]

Let $\kappa$ be a measure of concordance of degree at most
two, and let $b_\kappa$ be its canonical polarization. 
For each fixed $C\in\Cop$, define the map
$\Phi_C:\Cop\to\R$ by
$\Phi_C(D)=b_\kappa(C,D)$. By
Propositions~\ref{prop:canonical-polarization}
and~\ref{prop:characterization:b}, the map $\Phi_C$ is
continuous, affine, and $\preceq$-monotone. Therefore, by~Theorem~\ref{thm:affine-section-representation}, there exists a unique
measure $\nu_C^\kappa\in\Meas_w^+$ such that
\begin{align}\label{eq:Phi:representation}
  \Phi_C(D)
  =
  \Phi_C(\Pi)
  +\int_{(0,1)^2}(D-\Pi)\,\dd\nu_C^\kappa,
  \qquad D\in\Cop.
\end{align}
Let $C=\Pi$ in~\eqref{eq:Phi:representation}. 
Since $\Phi_{\Pi}(D)=\Phi_D(\Pi)$ and $\Phi_{\Pi}(\Pi)=b_{\kappa}(\Pi,\Pi)=\kappa(\Pi)=0$, we find $\Phi_D(\Pi)=
    \int_{(0,1)^2}(D-\Pi)\,\dd\nu_\Pi^\kappa$.
 Substituting this into~\eqref{eq:Phi:representation} yields the representation~\eqref{eq:anchored:field:representation}.
 Since $\nu_C^\kappa$ is unique for each $C\in \Cop$, the field $\nu^\kappa$ is also unique.
 Note that Condition~\ref{cond:positive:w:field} is automatically satisfied.
 We now check the other conditions for the field $\nu^\kappa$ to be admissible. 

 We first show the affinity~\ref{cond:affine}.
Fix $C,D\in\Cop$ and $t\in[0,1]$, and put
  $C_t:=(1-t)C+tD$. Separate affinity of $b_\kappa$ gives $\Phi_{C_t}(E)
    =
    (1-t)\Phi_C(E)+t\Phi_D(E)$, $E\in \Cop$.
  It then follows that
  \begin{align*}
    \Phi_{C_t}(E)-\Phi_{C_t}(\Pi)
    &=
    (1-t)\bigl\{\Phi_C(E)-\Phi_C(\Pi)\bigr\} +t\bigl\{\Phi_D(E)-\Phi_D(\Pi)\bigr\}\\
    &=
    \int_{(0,1)^2}(E-\Pi)\,
      \dd\bigl((1-t)\nu_C^\kappa+t\nu_D^\kappa\bigr),\qquad E\in \Cop.
  \end{align*}
  Since the measure
  $(1-t)\nu_C^\kappa+t\nu_D^\kappa$ belongs to
  $\Meas_w^+$, uniqueness of the representing measure for
  $\Phi_{C_t}$ yields $\nu_{(1-t)C+tD}^\kappa
    =
    (1-t)\nu_C^\kappa+t\nu_D^\kappa$.

    We next show~\ref{cond:reciprocity}.
    By~\eqref{eq:anchored:field:representation} and symmetry of $b_\kappa$, we have that
  \begin{align*}
    \int_{(0,1)^2}(C-\Pi)\,\dd\nu_\Pi^\kappa
    +\int_{(0,1)^2}(D-\Pi)\,\dd\nu_C^\kappa
    =
    \int_{(0,1)^2}(D-\Pi)\,\dd\nu_\Pi^\kappa
    +\int_{(0,1)^2}(C-\Pi)\,\dd\nu_D^\kappa.
  \end{align*}
  Rearranging this identity yields~\eqref{eq:reciprocity:identity}.

To verify~\ref{cond:D4:equivariance}, we first record the identity
\begin{align}\label{eq:copula-action-function}
  C^\gamma-\Pi
  &=
  \eps(\gamma)(C-\Pi)\circ\gamma^{-1},
  \qquad C\in\Cop.
\end{align}
By the action law and the multiplicativity of $\eps$, it suffices to
verify this identity for the generators $\pi$ and $\sigma_1$ of
$\mathcal D_4$. Since $\eps(\pi)=1$ and $\eps(\sigma_1)=-1$,
\begin{align*}
  (C^\pi-\Pi)(u,v)
  &=
  C(v,u)-vu
  =
  ((C-\Pi)\circ\pi)(u,v), \\
  (C^{\sigma_1}-\Pi)(u,v)
  &=
  v-C(1-u,v)-uv \\
  &=
  -\{C(1-u,v)-(1-u)v\}\\
  &=
  -\{(C-\Pi)\circ\sigma_1\}(u,v).
\end{align*}
Thus \eqref{eq:copula-action-function} holds for every
$\gamma\in\mathcal D_4$.

Now fix $C\in\Cop$ and $\gamma\in\mathcal D_4$. We show that
$\gamma_\#\nu_C^\kappa$ is the representing measure associated with
$\Phi_{C^\gamma}$, which directly implies~\ref{cond:D4:equivariance} by the uniqueness of the representing measure. 

First, we have that $\gamma_\#\nu_C^\kappa\in \Meas_w^+$. 
Indeed, since $w(u,v)=\min\{u,v,1-u,1-v\}$ is $\mathcal D_4$-invariant in the sense that $w\circ \gamma = w$, we have
\begin{align*}
  \int_{(0,1)^2}w\,\dd(\gamma_\#\nu_C^\kappa)
  =
  \int_{(0,1)^2}w\circ\gamma\,\dd\nu_C^\kappa =
  \int_{(0,1)^2}w\,\dd\nu_C^\kappa
  <\infty.
\end{align*}

We next consider $\Phi_{C^\gamma}$.
Let
$D\in\Cop$. Since $(D^{\gamma^{-1}})^\gamma=D$, we have from Proposition~\ref{prop:characterization:b}~\ref{cond:B1} that
\begin{align*}
  b_\kappa(C^\gamma,D)
  =
  b_\kappa\bigl(C^\gamma,
    (D^{\gamma^{-1}})^\gamma\bigr) =
  \eps(\gamma)b_\kappa(C,D^{\gamma^{-1}}).
\end{align*}
Similarly, since $\Pi^\gamma=\Pi$, we have
 $b_\kappa(C^\gamma,\Pi)
  =
  \eps(\gamma)b_\kappa(C,\Pi)$.
  Together with~\eqref{eq:Phi:representation}, we obtain
\begin{align*}
  b_\kappa(C^\gamma,D)-b_\kappa(C^\gamma,\Pi)
  =
  \eps(\gamma)
  \bigl\{
    b_\kappa(C,D^{\gamma^{-1}})
    -
    b_\kappa(C,\Pi)
  \bigr\} =
  \eps(\gamma)
  \int_{(0,1)^2}
    (D^{\gamma^{-1}}-\Pi)\,\dd\nu_C^\kappa.
\end{align*}
Applying \eqref{eq:copula-action-function} to $D$ and
$\gamma^{-1}$ gives $D^{\gamma^{-1}}-\Pi
  =
  \eps(\gamma^{-1})(D-\Pi)\circ\gamma$.
Because $\eps$ takes values in $\{-1,1\}$ and is multiplicative, we have $\eps(\gamma^{-1})
  =
  \eps(\gamma)$.
Therefore,
\begin{align*}
  \eps(\gamma)(D^{\gamma^{-1}}-\Pi)
  =
  \eps(\gamma)^2(D-\Pi)\circ\gamma=
  (D-\Pi)\circ\gamma.
\end{align*}
It thus follows that
\begin{align*}
  b_\kappa(C^\gamma,D)-b_\kappa(C^\gamma,\Pi)
  =
  \int_{(0,1)^2}(D-\Pi)\circ\gamma\,
    \dd\nu_C^\kappa =
  \int_{(0,1)^2}(D-\Pi)\,
    \dd(\gamma_\#\nu_C^\kappa).
\end{align*}
Combining this with~\eqref{eq:Phi:representation}, we have
\begin{align*}
  \int_{(0,1)^2}(D-\Pi)\,
    \dd\nu_{C^\gamma}^\kappa
  &=
  \int_{(0,1)^2}(D-\Pi)\,
    \dd(\gamma_\#\nu_C^\kappa),\qquad D \in \Cop.
\end{align*}
By the uniqueness assertion of the representing measure, we conclude $\nu_{C^\gamma}^\kappa
  =
  \gamma_\#\nu_C^\kappa$.

Continuity of the diagonal map~\ref{cond:diagonal:continuity} follows directly from 
\begin{align*}
  \zeta_{\nu^\kappa}(C)
  =
  \int_{(0,1)^2}(C-\Pi)\,\dd\nu_\Pi^\kappa
  +
  \int_{(0,1)^2}(C-\Pi)\,\dd\nu_C^\kappa  =
  b_\kappa(C,C)
  =
  \kappa(C).
\end{align*}

Finally,~\ref{cond:normalization} can be checked directly from~\eqref{eq:anchored:field:representation} and $b_\kappa(M,M)=\kappa(M)=1$.
\end{proof}

\subsection{Sufficiency part of
Theorem~\ref{thm:measure:field:characterization}}
\label{sec:sufficiency:uniqueness}

\begin{proof}[Proof of
Theorem~\ref{thm:measure:field:characterization}%
~\ref{item:characterization:only:if}]
Let $\nu=(\nu_C)_{C\in\Cop}$ be an admissible field, and define
$\beta_\nu:\Cop\times\Cop\to\R$ by
\begin{align*}
  \beta_\nu(C,D)
  &=
  \int_{(0,1)^2}(C-\Pi)\,\dd\nu_\Pi
  +
  \int_{(0,1)^2}(D-\Pi)\,\dd\nu_C.
\end{align*}
Since $|E-\Pi|\leq w$ for every $E\in\Cop$ and
$\nu_C\in\Meas_w^+$ for every $C\in\Cop$, both integrals are finite.
Note that $\kappa_\nu(C)=
  \beta_\nu(C,C)$, $C\in\Cop$.

We first establish that $\beta_\nu$ is separately affine. Fix
$C,D,E\in\Cop$ and $t\in[0,1]$. By the affinity of the field,
\begin{align*}
  \beta_\nu((1-t)C+tE,D)
  &=
  \int_{(0,1)^2}
    \left\{(1-t)C+tE-\Pi\right\}\,\dd\nu_\Pi 
  +
  \int_{(0,1)^2}(D-\Pi)\,
    \dd\left\{(1-t)\nu_C+t\nu_E\right\} \\
  &=
  (1-t)\beta_\nu(C,D)+t\beta_\nu(E,D).
\end{align*}
Thus $\beta_\nu$ is affine in its first argument. Its affinity in the
second argument follows directly from its definition.

Next, we show that $\beta_\nu$ is symmetric.
The reciprocity condition~\ref{cond:reciprocity} implies that 
\begin{align*}
  \beta_\nu(C,D)-\beta_\nu(D,C)
  &=
  \int_{(0,1)^2}(D-\Pi)\,
    \dd(\nu_C-\nu_\Pi) 
  -
  \int_{(0,1)^2}(C-\Pi)\,
    \dd(\nu_D-\nu_\Pi)=0,
\end{align*}
and thus $\beta_\nu$ is symmetric.

Symmetry and separate-affinity imply that $\kappa_\nu$ has degree at most two.
Indeed, for $C_t=(1-t)C+tD$, $t\in[0,1]$, we have $\kappa_\nu(C_t)
  =
  (1-t)^2\kappa_\nu(C)
  +2t(1-t)\beta_\nu(C,D)
  +t^2\kappa_\nu(D)$, which is a polynomial in $t$ of degree at most two.

We next verify the axiomatic properties in Definition~\ref{def:moc}. 

To show~\ref{item:pi:sym} and~\ref{item:sigma:sym}, we prove $\kappa_\nu(C^\gamma)
  =
  \eps(\gamma)\kappa_\nu(C)$ for $C\in\Cop$, $\gamma\in\mathcal D_4$.
Since
$\Pi^\gamma=\Pi$, Condition~\ref{cond:D4:equivariance} gives
$\gamma_\#\nu_\Pi
  =
  \nu_{\Pi^\gamma}
  =
  \nu_\Pi$, $\gamma\in\mathcal D_4$.
  It thus follows from~\eqref{eq:copula-action-function} and $\nu_{C^\gamma}=\gamma_\#\nu_C$ that
\begin{align*}
  \beta_\nu(C^\gamma,D^\gamma)
  &=
  \int_{(0,1)^2}(C^\gamma-\Pi)\,
    \dd(\gamma_\#\nu_\Pi)
  +
  \int_{(0,1)^2}(D^\gamma-\Pi)\,
    \dd(\gamma_\#\nu_C) \\
  &=
  \int_{(0,1)^2}
    \{(C^\gamma-\Pi)\circ\gamma\}\,\dd\nu_\Pi
  +
  \int_{(0,1)^2}
    \{(D^\gamma-\Pi)\circ\gamma\}\,\dd\nu_C \\
  &=
  \eps(\gamma)\beta_\nu(C,D).
\end{align*}
Taking $D=C$ yields the desired equality $\kappa_\nu(C^\gamma)
  =
  \eps(\gamma)\kappa_\nu(C)$, $C\in\Cop,\ \gamma\in\mathcal D_4$, implying~\ref{item:pi:sym} and~\ref{item:sigma:sym}.

We next verify the monotonicity axiom~\ref{item:monotone}.
If $D_1\preceq D_2$, then, for every
$C\in\Cop$,
\begin{align*}
  \beta_\nu(C,D_2)-\beta_\nu(C,D_1)
  &=
  \int_{(0,1)^2}(D_2-D_1)\,\dd\nu_C
  \geq 0.
\end{align*}
Therefore, $\beta_\nu$ is non-decreasing in
its second argument. 
By symmetry, $\beta_\nu$ is also non-decreasing in its first argument.
Consequently, if $C\preceq D$, then
\begin{align*}
  \kappa_\nu(C)
  =
  \beta_\nu(C,C)
  \leq
  \beta_\nu(C,D)
  \leq
  \beta_\nu(D,D)
  =
  \kappa_\nu(D).
\end{align*}

Condition~\ref{item:M} is a direct consequence of Condition~\ref{cond:normalization}.
Continuity~\ref{item:conv} also follows immediately from~\ref{cond:diagonal:continuity} and $\zeta_\nu=\kappa_\nu$.
Therefore, we have now shown that $\kappa_\nu$ is a measure of concordance of degree at most two.

The map $\beta_\nu$ is symmetric and separately
affine, and its diagonal is $\kappa_\nu$. Therefore, 
Proposition~\ref{prop:canonical-polarization}~(iv) gives $\beta_\nu
  =
  b_{\kappa_\nu}$.
 By comparing the definition of $\beta_\nu$ and~\eqref{eq:anchored:field:representation}, the uniqueness of the admissible field established in Theorem~\ref{thm:measure:field:characterization}~\ref{item:characterization:if} leads to $ \nu
  =
  \nu^{\kappa_\nu}$.
  This completes the proof.
\end{proof}

\section{Proofs for Section~\ref{sec:constructions}}

\subsection{Proposition~\ref{prop:operator-construction}}

\begin{proof}[Proof of Proposition~\ref{prop:operator-construction}]

We first check that $\nu^{\Theta}$ is an admissible field.
Condition~\ref{cond:positive:w:field} follows from non-negativity and finiteness of $\nu_C^\Theta$ and boundedness of $w$.
Condition~\ref{cond:affine} follows from the affinity of $\Theta$.

To show~\ref{cond:reciprocity}, define
\begin{align*}
  \beta_\Theta(C,D)
  =
  \int_{(0,1)^2}(C-\Pi)\,\dd\nu_\Pi^\Theta
  +\int_{(0,1)^2}(D-\Pi)\,\dd\nu_{C}^\Theta.
  \end{align*}
  Then
\begin{align*}
  \beta_\Theta(C,D)
  &=
  4\int_{(0,1)^2}(C-\Pi)\,\dd\mu_\Pi
  +4\int_{(0,1)^2}(D-\Pi)\,\dd\mu_{\Theta(C)}\\
  &=
  Q(\Pi,C)-Q(\Pi,\Pi)
  +Q(\Theta(C),D)-Q(\Theta(C),\Pi).
\end{align*}
Note that $Q(\Pi,\Pi)=0$ and $  Q(\Theta(C),\Pi)
  =
  Q(C,\Theta(\Pi))
  =
  Q(C,\Pi)
  =
  Q(\Pi,C)
$ by symmetry of $Q$ and
Conditions~\ref{cond:T-fixed} and~\ref{cond:T-self-adjoint}.
Therefore, we have $
  \beta_\Theta(C,D)
  =
  Q(\Theta(C),D).
$
Condition~\ref{cond:T-self-adjoint} and symmetry of $Q$ now give
\begin{align*}
  \beta_\Theta(C,D)
  &=
  Q(\Theta(C),D)
  =
  Q(C,\Theta(D))
  =
  Q(\Theta(D),C)
  =
  \beta_\Theta(D,C).
\end{align*}
Thus $\beta_\Theta$ is symmetric.
The equality $\beta_\Theta(D,C)=\beta_\Theta(C,D)$ directly leads to the reciprocity condition for $\nu^\Theta$.

Condition~\ref{cond:T-equivariant} implies Condition~\ref{cond:D4:equivariance} since  
$  \nu_{C^\gamma}^\Theta
  =
  4\mu_{\Theta(C^\gamma)}
  =
  4\mu_{\Theta(C)^\gamma}
  =
  \gamma_\#\nu_C^\Theta.
$
To show Condition~\ref{cond:diagonal:continuity}, note that the diagonal map can be written as $\zeta_{\nu^\Theta}(C)
  =
  \beta_\Theta(C,C)
  =
  Q(\Theta(C),C).
$
Since $\Theta$ is continuous by assumption and $Q$ is continuous by Theorem 4.3 of~\citet{fuchs2016biconvex}, we have that $\zeta_{\nu^\Theta}$ is continuous.
Finally,~\ref{cond:normalization} follows from $\beta_\Theta(M,M)=Q(\Theta(M),M)
  =
  Q(M,M)
  =
  1.
$
Consequently, $\nu^\Theta$ is an
admissible field.

Theorem~\ref{thm:measure:field:characterization}~\ref{item:characterization:only:if} now shows that
$\kappa_\Theta$ is a measure of concordance of degree at most two, with $\nu^\Theta$ being its admissible field and $
  b_{\kappa_\Theta}(C,D)
  =
  \beta_\Theta(C,D)
  =
  Q(\Theta(C),D).
$
Therefore, it remains to show that $\kappa_\Theta$ has exact degree two.
To check this,  consider $C_t
  =
  (1-t)\Pi+tM$, $t \in [0,1]$.
  Affinity and Condition~\ref{cond:T-fixed} give $\Theta(C_t)=C_t$.
Together with $Q(\Pi,M)=1/3$, we have that
\begin{align*}
  \kappa_\Theta(C_t)
  =Q(\Theta(C_t),C_t)=
  Q(C_t,C_t)=
  2t(1-t)Q(\Pi,M)+t^2Q(M,M)=
  \frac{2t+t^2}{3},
\end{align*}
which is a polynomial of degree exactly two.
\end{proof}

\subsection{Proposition~\ref{prop:group:averaging}}
\begin{proof}[Proof of Proposition~\ref{prop:group:averaging}]

We first list some properties of $\mathcal D_4$ necessary in this proof.
Let
\begin{align*}
  \mathcal H_M
  &=
  \{\gamma\in\mathcal D_4:M^\gamma=M\}
  =
  \{\mathrm{id},\pi,\varrho,\pi\circ \varrho\}=
  \{\pi^a \circ \varrho^b:a,b\in\{0,1\}\},
\end{align*}
with $\pi^0=\varrho^0=\mathrm{id}$.
Since $\pi\circ\varrho=\varrho\circ\pi$ and
$\pi^2=\varrho^2=\mathrm{id}$, the subgroup
$\mathcal H_M$ is abelian.
Note that $\mathcal D_4$ itself is not since $\sigma_1\circ \pi \neq \pi \circ \sigma_1$.

The subgroup $\mathcal H_M$ is normal in $\mathcal D_4$, that is,  $\delta^{-1}\mathcal H_M\delta
  =
  \mathcal H_M$, $\delta\in\mathcal D_4$.
This is indeed true for $\delta\in\mathcal H_M$ due to its commutativity.
Note that $\mathcal D_4=\mathcal H_M \cup \sigma_1 \mathcal H_M$. 
Therefore, any $\delta\in\mathcal D_4\setminus\mathcal H_M$ is of the form $\delta=\sigma_1\circ h$ for $h \in \mathcal H_M$.
In addition, it can be directly checked that $\sigma_1^{-1}\mathcal H_M \sigma_1 =\mathcal H_M$.
Therefore, 
\begin{align*}
  \delta^{-1}\mathcal H_M\delta
  =h^{-1}\circ \sigma_1^{-1}
  \mathcal H_M\sigma_1\circ h=h^{-1}\mathcal H_M h=\mathcal H_M.
\end{align*}
Consequently, $\mathcal H_M$ is a normal subgroup of $\mathcal D_4$.

Suppose first that $w_\pi
  =
  w_{\pi\circ \varrho}$ and $w_{\sigma_1}
  =
  w_{\sigma_2}
  =
  w_{\sigma_1\circ\pi}
  =
  w_{\pi\circ\sigma_1}
  =
  0$.
Then $\Theta_{\bm w}(C)
  =
  \sum_{\gamma\in\mathcal H_M}
  w_\gamma C^\gamma$.
We first verify Condition~\ref{cond:T-equivariant}. 
Fix $C\in \Cop$ and $\delta\in \mathcal D_4$. 
By normality of $\mathcal H_M$ in $\mathcal D_4$ and affinity of the map $C\mapsto C^\delta$, we have that
\begin{align*}
  \{\Theta_{\bm w}(C^\delta)\}^{\delta^{-1}}
  &=
  \sum_{\gamma\in\mathcal H_M}
  w_\gamma C^{\delta^{-1}\circ\gamma\circ\delta}
  =
  \sum_{\eta\in\mathcal H_M}
  w_{\delta\circ \eta\circ \delta^{-1}}C^\eta
  =
  \sum_{\eta\in\mathcal H_M}
  w_\eta C^\eta
  =
  \Theta_{\bm w}(C),
\end{align*}
where $\eta=\delta^{-1}\circ\gamma\circ\delta$. Applying
$\delta$ to both sides gives~Condition~\ref{cond:T-equivariant}.
Condition~\ref{cond:T-fixed} follows immediately by noticing that $\Pi^\gamma=\Pi$ and $M^\gamma=M$ for $\gamma \in \mathcal H_M$.
Finally, to show Condition~\ref{cond:T-self-adjoint}, note that separate affinity of $Q$ and $\sum_{\gamma\in\mathcal H_M}w_\gamma=1$ imply that $Q(\Theta_{\bm w}(C),D)
  =
  \sum_{\gamma\in\mathcal H_M}
  w_\gamma Q(C^\gamma,D)$ and $Q(C,\Theta_{\bm w}(D))=
  \sum_{\gamma\in\mathcal H_M}
  w_\gamma Q(C,D^\gamma)$.
  Consequently, Condition~\ref{cond:T-self-adjoint} holds by showing that $Q(C^\gamma,D)=Q(C,D^\gamma)$ for every $C,D\in \Cop$ and $\gamma\in \mathcal H_M$.
  This is trivially true for $\gamma=\mathrm{id}$. For the other three elements of $\mathcal H_M$, we can directly check as follows:
\begin{align*}
  Q(C^\pi,D)
  &=
  4\int_{[0,1]^2}C(v,u)\,\dd\mu_D(u,v)-1\\
&  =
  4\int_{[0,1]^2}C(u,v)\,\dd\mu_{D^\pi}(u,v)-1\\
&  =
  Q(C,D^\pi);\\
  Q(C^\varrho,D)
  &=
  4\int_{[0,1]^2}
  \bigl\{u+v-1+C(1-u,1-v)\bigr\}
  \,\dd\mu_D(u,v)-1
  \\
  &=
 4\int_{[0,1]^2}
  (u+v-1)
  \,\dd\mu_D(u,v)
  + 4\int_{[0,1]^2}
  C(1-u,1-v)
  \,\dd\mu_D(u,v)-1\\
  &= 0 + 
  4\int_{[0,1]^2}
  C(u,v)\,\dd\mu_{D^\varrho}(u,v)-1
  \\
  &=
  Q(C,D^\varrho),\\
  Q(C^{\pi\circ \varrho},D)
  &=
  4\int_{[0,1]^2}
  \bigl\{u+v-1+C(1-v,1-u)\bigr\}
  \,\dd\mu_D(u,v)-1
  \\
  &= 4\int_{[0,1]^2}
  (u+v-1)
  \,\dd\mu_D(u,v)
  +4\int_{[0,1]^2}
C(1-v,1-u)
  \,\dd\mu_D(u,v)-1\\
  &=0+
  4\int_{[0,1]^2}
  C(u,v)\,\dd\mu_{D^{\pi\circ \varrho}}(u,v)-1
  \\
  &=
  Q(C,D^{\pi\circ \varrho}).
\end{align*}

Conversely, suppose that $\Theta_{\bm w}$ satisfies
Conditions~\ref{cond:T-equivariant}--\ref{cond:T-self-adjoint}. Set $s
  =
  \sum_{\gamma\notin\mathcal H_M}w_\gamma$.
  Since $M^\gamma=W$ for $\gamma\in\mathcal D_4\setminus\mathcal H_M$, we have $\Theta_{\bm w}(M)
  =
  (1-s)M+sW$.
Condition~\ref{cond:T-fixed} and $M\neq W$ imply $s=0$. Since all
weights are non-negative, this is equivalent to $w_{\sigma_1}
  =
  w_{\sigma_2}
  =
  w_{\sigma_1\circ\pi}
  =
  w_{\pi\circ\sigma_1}
  =
  0$.
Hence $\Theta_{\bm w}(C)
  =
  \sum_{\gamma\in\mathcal H_M}w_\gamma C^\gamma$.

It remains to show $w_\pi=w_{\pi\circ \varrho}$.
Let $\delta\in\mathcal D_4$. Then Condition
\ref{cond:T-equivariant} gives
\begin{align*}
  \Theta_{\bm w}(C)
  =
  \{\Theta_{\bm w}(C^\delta)\}^{\delta^{-1}}
  =
  \sum_{\gamma\in\mathcal H_M}
  w_\gamma C^{\delta^{-1}\circ\gamma\circ\delta}
  =
  \sum_{\eta\in\mathcal H_M}
  w_{\delta\circ \eta\circ \delta^{-1}}C^\eta,
\end{align*}
where $\eta=\delta^{-1}\circ \gamma\circ \delta$. Therefore,
\begin{align*}
  \sum_{\gamma\in\mathcal H_M}
  w_\gamma C^\gamma
  &=
  \sum_{\gamma\in\mathcal H_M}
  w_{\delta\circ\gamma\circ\delta^{-1}}C^\gamma,
  \qquad
  C\in\Cop,\quad
  \delta\in\mathcal D_4.
\end{align*}
Taking $\delta=\sigma_1$, we obtain
\begin{align}
  (w_\pi-w_{\pi\circ \varrho})
  (C^\pi-C^{\pi\circ \varrho})
  &=
  0,
  \qquad
  C\in\Cop.
  \label{eq:diagonal-weight-separation}
\end{align}

For $C$ in~\eqref{eq:diagonal-weight-separation}, let
$C=C_\#$ be the $4\times4$ checkerboard copula with density
\begin{align*}
  c_\#(u,v)
  &=
  4\sum_{i=1}^4
  \id_{R_{i,p(i)}}(u,v),
  \qquad
  p=(p(1),\dots,p(4))=(1,3,4,2),
\end{align*}
where $R_{ij}=I_i\times I_j$, $i,j\in\{1,2,3,4\}$, and
$I_i=((i-1)/4,i/4]$.
Note that $\mu_{C_\#}$ assigns mass $1/4$ to each of
$R_{11}$, $R_{23}$, $R_{34}$, and $R_{42}$, and zero mass to all
other cells.
Since $\mu_{C_\#^\pi}(R_{24})=1/4$ and $\mu_{C_\#^{\pi\circ \varrho}}(R_{24})=0$, it follows from~\eqref{eq:diagonal-weight-separation} that
\begin{align*}
  0
  =
  (w_\pi-w_{\pi\circ \varrho})
  \left\{
    \mu_{C_\#^\pi}(R_{24})
    -
    \mu_{C_\#^{\pi\circ \varrho}}(R_{24})
  \right\}
  =
  \frac{w_\pi-w_{\pi\circ \varrho}}4.
\end{align*}
Thus $w_\pi=w_{\pi\circ \varrho}$. 
This completes the proof.
\end{proof}
\end{document}